\documentclass[12pt,a4paper]{amsart}
\usepackage{latexsym,amsfonts,amsthm,amsmath,mathrsfs,amssymb}
\usepackage[british]{babel}
\usepackage[dvips]{color}
\usepackage[utf8]{inputenc}
\usepackage[T1]{fontenc}
\usepackage[dvips,final]{graphics}
\usepackage{xcolor}
\usepackage{mathrsfs}
\usepackage{mathtools} 
\usepackage{breqn}
\usepackage{epstopdf}
\usepackage{verbatim}
\usepackage{amscd}
\usepackage{hyperref}
\usepackage[english,capitalize]{cleveref}
\usepackage{aliascnt}

\usepackage{todonotes}

\makeatletter
\def\newaliasedtheorem#1[#2]#3{
	\newaliascnt{#1@alt}{#2}
	\newtheorem{#1}[#1@alt]{#3}
	\expandafter\newcommand\csname #1@altname\endcsname{#3}
}
\makeatother

\theoremstyle{plain}

\newtheorem{theorem}{Theorem}[section]
\newaliasedtheorem{prop}[theorem]{Proposition}
\newaliasedtheorem{lem}[theorem]{Lemma}
\newaliasedtheorem{coroll}[theorem]{Corollary}

\theoremstyle{definition}
\newaliasedtheorem{defi}[theorem]{Definition}
\newaliasedtheorem{quest}[theorem]{Question}
\newaliasedtheorem{fact}[theorem]{Fact}

\theoremstyle{remark}
\newaliasedtheorem{rem}[theorem]{Remark}
\newaliasedtheorem{exa}[theorem]{Example}

\newcommand{\R}{\mathbb{R}}
\newcommand{\N}{\mathbb{N}}

\newcommand{\HH}{\mathbb{H}}

\newcommand{\eps}{\varepsilon}
\let\altphi\phi
\let\phi\varphi
\let\varphi\altphi
\let\altphi\undefined

\newcommand{\average}{{\mathchoice {\kern1ex\vcenter{\hrule height.4pt
width 6pt
depth0pt} \kern-9.7pt} {\kern1ex\vcenter{\hrule height.4pt width 4.3pt
depth0pt}
\kern-7pt} {} {} }}

\address{\textsc{Daniela Di Donato}: 
Dipartimento di Ingegneria Industriale e Scienze Matematiche, Via Brecce Bianche, 12 60131 Ancona, Universit\'a Politecnica delle Marche.}
\email{daniela.didonato@unitn.it}
\address{\textsc{Enrico Le Donne}: 
	Dipartimento di Matematica, Università di Pisa, Largo B. Pontecorvo 5, 56127 Pisa, Italy \\ University of Jyväskylä, Department of Mathematics and Statistics, P.O. Box (MaD), FI-40014, Finland \\ Department of Mathematics, University of Fribourg, Chemin du Musée 23, 1700 Fribourg, Switzerland}
\email{enrico.ledonne@unifr.ch}

\title{Intrinsically Lipschitz sections and applications to metric groups}

\date{\today}

\author{ Daniela Di Donato and Enrico Le Donne}

\begin{document}

\begin{abstract}
		We introduce a notion of  intrinsically Lipschitz graphs in the context of metric spaces.
		This is a broad generalization of what in  Carnot groups has been considered by
		  Franchi, Serapioni, and Serra Cassano, and later by many others. 
		  We proceed by focusing our attention on the graphs as subsets of a metric space given by the image of a section of a quotient map and we require an intrinsically Lipschitz condition. We shall not have any function on a topological product, not we shall consider a metric on the base of the quotient map. 
		 Our  results are: an Ascoli-Arzel\`a compactness theorem, an Ahlfors regularity theorem, and some extension theorems for partially defined intrinsically Lipschitz sections.
		 Known results by Franchi, Serapioni, and Serra Cassano, and by Vittone will be our corollaries.
		
	\end{abstract}

\maketitle 
\tableofcontents
\newpage

\section{Introduction}
Nowadays, in the setting of subRiemannian Carnot groups, there is a rich theory of geometric analysis: from geometric measure theory to partial differential equations \cite{BLU, CDPT, Corni19, CorniMagnani, Mag13, Pansu, NY, Rigot, SC16}.  
Objects that play the role of Lipschitz submanifolds are the intrinsically Lipschitz graphs, originally introduced by Franchi, Serapioni, and Serra Cassano, in order to have an adapted notion of rectifiability of (boundaries of) finite-perimeter sets \cite{ASCV, AmbrosioKirchheimRect, AKLD, ADDDLD, biblioAM2, ArenaSerapioni, BCSC, BSC, BSCgraphs, MR3992573, MR2263950, DD19, DD20, FMS14, FSSC06, FSSC07, FSSC11, JNGV, MV, MR3984100}.

The purpose of the present article is to axiomatize the notion of intrinsically Lipschitz graphs to the setting of metric spaces.
We shall prove that some of the geometric results from the last decade are valid in a broad generality. Our setting is the following. We have a metric space $X$, a topological space $Y$, and a 
quotient map $\pi:X\to Y$, meaning
continuous, open, and surjective.
The standard example for us is when $X$ is a metric group $G$ (meaning a topological group $G$   equipped with a left-invariant distance that induces the  topology), for example a subRiemannian Carnot group (see  \cite{ledonne_primer}), 
and $Y$ is the space of left cosets $G/H$, where 
$H<G$ is a  closed subgroup and $\pi:G\to G/H$ is the projection modulo $H$, $g\mapsto gH$. The inexperienced reader may find a more basic example in Example~\ref{baby_example}.

The objects of study in this paper are the following type of maps, which we shall prove generalize the ones in \cite{FSSC, FSSC03, MR2032504}, see also  \cite{SC16, FS16}.

\begin{defi}\label{def_ILS} Given a quotient map   $\pi:X\to Y$  
between a  metric space $X$ and  a topological space $Y$, we say that a map $\phi:Y\to X$ is an {\em intrinsically Lipschitz section of $\pi$ with constant $L$},  with $L\in[1,\infty)$, if
\begin{equation}
\pi \circ \phi =\mbox{id}_Y,
\end{equation}
and
\begin{equation}\label{eq:def:ILS}
d(\phi (y_1), \phi (y_2)) \leq L d(\phi (y_1), \pi ^{-1} (y_2)), \quad \mbox{for all } y_1, y_2 \in Y.
\end{equation}
Here $d$ denotes both the distance between points on $X$, and, as usual, for a subset $A\subset X$ and a point $x\in X$, we have
$d(x,A):=\inf\{d(x,a):a\in A\}$.
\end{defi}

We shall also briefly say that a map as in Definition~\ref{def_ILS} is an {\em intrinsically $L$-Lipschitz  section} or that it is an {\em intrinsically  Lipschitz  section} if there is no need to specify $\pi$, nor $L$. In Section~\ref{sec:equiv_def}, we will give other characterizations of intrinsically  Lipschitz  sections, also in terms of the fact that the images have trivial intersection with some particular subsets of $X$, which unlike in the Carnot setting don't have anymore a structure of cones, because there is no dilation/homogeneous structure assumed on $X$. 

We shall call the image $\phi(Y)$ of some intrinsically Lipschitz section $\phi:Y\to X$ an {\em intrinsically Lipschitz graph}.
We stress that, even if the set $\phi(Y)$ is parametrized  by $Y$, the geometric regularity of $\phi(Y)$ depends only on the ambient distance in $X$, and not on the one of $Y$, which a priori  we haven't metrized. 
 In particular, the set $\phi(Y)$ may not be biLipschitz   equivalent to $Y$.
 In some situations we might have a natural way of metrizing $Y$, but it might not be the case that   $\pi$ becomes a Lipschitz quotient,  e.g. a submetry (see Section~\ref{sec:prel} for these definitions). In the case $
 \pi$ is a Lipschitz quotient, the results trivialize, since in this case being intrinsically Lipschitz  is equivalent to being a biLipschitz embedding, see Proposition~\ref{tivial_Lip_quo}. In the context of groups, one has a Lipschitz quotient    when one takes a normal subgroup $N\lhd G$ of   a metric Lie group  $G$  and $\pi:G\to G/N$.
 
 In case $G$ is a Carnot group that can be written as product of two complementary homogeneous subgroups $G=H_1 \cdot H_2$ then our notion of intrinsically Lipschitz graph coincides with the one given by Franchi-Serapioni-SerraCassano, see Section~\ref{sec:various_notions}. However, we stress that in the case of an arbitrary decomposition $G=H_1 \cdot H_2$, with $H_1$ and $H_2$   not necessarily homogeneous,  
 one could naturally identify $G/H_2$ with $H_1$ but, first, if $H_2$ is not normal, then the cosets of $H_2$ are not parallel within $G$, and even if $H_2$ is   normal, then the projection onto $H_1$ may not be a Lipschitz quotient. It is crucial to remark that in the case of Franchi, Serapioni, and Serra Cassano the homogeneity assumption will imply such a property, see Section~\ref{sec:various_notions}. The aim of our work is to provide generalizations to the metric setting of results known in the case of homogeneous decomposition of groups.
 
 \medskip
 
  The first series of our results is about the equicontinuity of intrinsically Lipschitz sections with uniform constant and consequently a compactness property \`a la Ascoli-Arzel\'a. Be aware that, in the literature on metric spaces, another term for    boundedly compact is {\em proper}.

 \begin{theorem}[Equicontinuity and Compactness Theorem]
 \label{thm1}
 Let $\pi:X\to Y$ be a quotient map 
between a  metric space $X$ and  a topological space $Y$.   

\noindent{\rm \bf(\ref{thm1}.i)} Every intrinsically Lipschitz section of $\pi$   is continuous.

    Next, assume in addition  that closed balls in  $X$ are compact (we say that $X$ is {\em boundedly compact}).

\noindent{\rm \bf(\ref{thm1}.ii)} For all
    $K' \subset Y$ compact,
  $L\geq 1 $,     $K \subset X$ compact, and  $y_0\in Y$
   the set
 \begin{equation*}\label{set2}
\{ \phi _{|_{K'}} : K' \to X \,| \, \phi :Y \to X \mbox{intrinsically $L$-Lipschitz  section of $\pi$}, \phi (y_0) \in K \}
\end{equation*}
is 
 equibounded, equicontinuous, and closed in the uniform-convergence topology.
 
 \noindent{\rm \bf(\ref{thm1}.iii)} For all $L\geq 1$,     $K \subset X$ compact, and  $y_0\in Y$
the set
 \begin{equation*}\label{set1}
\{ \phi : Y \to X \,| \, \phi \text{ intrinsically $L$-Lipschitz  section of $\pi$},
 \phi (y_0) \in K \}
\end{equation*}
  is compact with respect to  the topology of uniform convergence on compact sets.
  \end{theorem}

Without the assumption that $\pi$ is an open map,    intrinsically Lipschitz sections may not be continuous. See Example~\ref{rempiopen} for some pathological  intrinsically Lipschitz sections.

We stress, as similarly done before, that given two intrinsically Lipschitz sections 
$\phi_1, \phi_2:Y\to X$,  the two sets $\phi_1(Y)$ and $\phi_2(Y)$ may not be biLipschitz equivalent.
However, following the influential paper \cite{FS16}, we prove that, in the presence of a nice measure on $Y$, then if $\phi_1(Y)$ is an Ahlfors regular set then so is $\phi_2(Y)$.
The assumption of the existence of such a measure is necessary, see Example~\ref{exaAhlfors1711}.
 More  precisely, our next result is the following.

   \begin{theorem}[Ahlfors regularity, after Franchi-Serapioni-SerraCassano]\label{thm2}
   Let $\pi :X \to Y$ be a quotient map between a metric space $X$ and a topological space $Y$ such that there is a measure $\mu$ on $Y$ such that for every $r_0>0$ and every $x,x' \in X$ with $\pi (x)=\pi(x')$  there is $C>0$ such that 
   \begin{equation}\label{Ahlfors27ott.112}
\mu (\pi (B(x,r))) \leq C \mu (\pi (B(x',r))),\quad \forall  r\in (0,r_0).
\end{equation}
We also assume that there is  an intrinsically  Lipschitz section $\phi :Y \to X$  of $\pi$ such that $\phi (Y)$ is locally $Q$-Ahlfors regular with respect to  the measure $\phi_* \mu$, with $Q\in (0,\infty)$. 

  Then, for every intrinsically  Lipschitz section $\psi :Y \to X$ of $\pi$,   the set $ \psi (Y) $ is locally $Q$-Ahlfors regular with respect to  the measure $\psi_* \mu.$

    \end{theorem}
 Namely, in Theorem~\ref{thm2} the local Ahlfors $Q$-regularity of $\phi (Y)$ means that    the measure $\phi_* \mu$ is such that  for each point $x\in \phi (Y)$ there exist $r_0>0$ and $C>0$ so that
  \begin{equation}\label{Ahlfors_IN_N}
 C^{-1}r^Q\leq  \phi_* \mu \big( B(x,r) \cap \phi (Y)\big) \leq C r^Q, \qquad \text{ for all }r\in (0,r_0).
\end{equation}
The same inequality will hold for  $ \psi (Y) $ and $\psi_* \mu $ with a possibly different value of $C$. See Section~\ref{theoremAhlforsNEW}.
    
    \medskip
    
In some settings one has the necessity of describing ``regular'' submanifolds as zero sets of distinguished functions, which is useful to possibly extend partially defined objects. Generalizing a result of Vittone \cite{Vittone20}, we next show that this can also be done with intrinsically Lipschitz graphs, at least when we have a good control on fibers.
In next result, we say that a metric-space-valued  map $f$ on $X$ is $L$-{\em biLipschitz on  fibers} (of $\pi$)  if 
 on each fiber of $\pi$ it restricts to an $L$-biLipschitz homeomorphism.

\begin{theorem}[Extensions as level sets, after Vittone]\label{thm3}  
 Let $\pi:X\to Y$ be a quotient map 
between a  metric space $X$ and  a topological space $Y$. 

\noindent{\rm \bf(\ref{thm3}.i)}
If $Z$ is a metric space, $z_0\in Z$ and $f:X\to Z$ is  $L$-Lipschitz and $L$-biLipschitz on  fibers, with $L\geq 1$, then
  there exists an intrinsically $(1+L^2)$-Lipschitz section  $\phi : Y\to X $ of $\pi$ such that
\begin{equation}\label{equationluogoz0_intro}
\phi (Y)=f^{-1} (z_0).
\end{equation}

\noindent{\rm \bf(\ref{thm3}.ii)}
Vice versa, assume that $X$ is geodesic and that there exist $k, L\geq 1,$ $\rho : X \times X \to \R$ k-biLipschitz equivalent to the distance of $X,$ and $\tau :X \to \R $ k-Lipschitz and $k$-biLipschitz on  fibers such that  \begin{enumerate}
\item  for all $\tau_0\in \R $ the set $\tau ^{-1}(\tau_0)$ is an intrinsically $k$-Lipschitz graph of a section $\phi_{\tau_0}:Y\to X;$
\item for all $x_0\in \tau ^{-1} (\tau_0)$ the map $\delta _{\tau_0} :X\to \R, x \mapsto \delta _{\tau_0} (x) := \rho (x_0, \phi_{\tau_0}(\pi (x)))$ is $k$-Lipschitz on the set $\{|\tau-\tau_0| \leq k L \delta _{\tau_0} \}$.
\end{enumerate}
Let $Y'\subset Y$ be a set.
Then for every  intrinsically $L$-Lipschitz section  $\phi : Y'\to \pi^{-1}(Y') $  of $\pi|_{\pi^{-1}(Y')} :\pi^{-1}(Y') \to  Y' $, there exists a  map $ f:X\to \R$    that is $K$-Lipschitz and $K$-biLipschitz on  fibers, with $K:= 2k(Lk+2)$, such that     \begin{equation}\label{equationluogozeri_intro}
\phi (Y')\subseteq f^{-1} (0).
\end{equation}
In particular, each `partially defined' intrinsically Lipschitz graph $\phi (Y')$ is a subset of a `globally defined' intrinsically Lipschitz graph $ f^{-1} (0)$.
 \end{theorem}



We shall next apply our study to the case of metric groups and more specifically to the case of Carnot groups. We shall then see how our theorems give known results as immediate consequences.
Initially, we shall barely consider the case  of a metric   group $G$ and a closed subgroup $H$ of $G$. In such a way, we shall rephrase the notion of intrinsically Lipschitz section of  the quotient map $\pi : G\to G/H$. 
To have further geometric properties of intrinsically Lipschitz sections in groups,  we shall require a {\em splitting} $G=H_1\cdot H_2$ with $H_1, H_2$ closed subgroups.
As commonly done, writing $G=H_1\cdot H_2$ means that 
$G=\{h_1h_2: h_1\in H_1, h_2\in H_2\}$ and 
 $ H_1 \cap H_2 = \{1_G\}$.
For example, if in addition $H_1$ is a normal subgroup, then $G$
has  the structure of semidirect product $G=H_1 \rtimes H_2$. 
In the presence of a splitting $G=H_1\cdot H_2$, we have the two naturally defined projection maps $\pi_{H_i}:G\to H_i$. We stress that such maps may not be Lipschitz, not even when one of the groups is normal. An important setting, in which we have several equivalences, is when the map $\pi_{H_1}$ is {\em Lipschitz at $1_G$}, i.e.,
\begin{equation}\label{def:splitat1}
d(1_G, \pi_{H_1}(g) )\leq K d(1_G,g), \qquad \forall g\in G,
\end{equation}
where, with $1_G$ we denote the identity element of the group $G$.
See Section~\ref{sec:Lip_at_1} for other equivalent conditions for the Lipschitz property at $1_G$, especially in the case when $H_1$ is normal. For example, for us an important equivalent property is that the inclusion $H_1 \hookrightarrow G$ is an intrinsically Lipschitz section for the projection $G\to G/H_2$.

When we have a splitting $G=H_1\cdot H_2$ then the left-coset of $H_1$ are sections of the projection modulo $H_2$. In general, such sections may not be intrinsically Lipschitz (as we just said, they are if and only if $\pi_{H_1}$ is Lipschitz at $1_G$). We introduce a notion of being intrinsically Lipschitz with respect to these sections, see Definition~\ref{defwrtpsinew}. The two notions of intrinsically Lipschitz sections coincide under the assumption that the left cosets of $H_1$ are intrinsically Lipschitz (see Corollary~\ref{corol8.8}). 
These types of facts hold  in the general setting of metric spaces, as in the next result. 
 \begin{prop}\label{linkintrinsicocneelip}
   Let $X  $ be a metric space, $Y$ a topological space, $\pi :X \to Y$   a  quotient map, and $L\geq 1$.
Assume that   every point $x\in X$ is contained in the image of an intrinsically $L$-Lipschitz section $\psi_x$ for $\pi$.
 Then for every section $\phi :Y\to X$ of $\pi$ the following are equivalent:
   \begin{enumerate}
\item for some $L_1\geq1$ and for all $x\in\phi(Y)$ the section $\phi $ is intrinsically $L_1$-Lipschitz with respect to  $\psi_x$ at   $x$ (see  Definition~\ref{defwrtpsinew});
\item  the section $\phi $  is intrinsically $L_2$-Lipschitz.
\end{enumerate}
   \end{prop}
 
Next we make the link with the notion of intrinsically Lipschitz maps in the sense of Franchi, Serapioni, and Serra Cassano. Given a splitting $G=H_1\cdot H_2$,
for $\psi:H_1\to H_2$ we set
  $$\Gamma _ {\psi} := \{n\psi(n)\,:\, n \in H_1\}.$$ We say that $\psi$ is an 
{\em  intrinsically Lipschitz map in the FSSC sense}  if  exists $K>0$ such that
\begin{equation}
d(1_G, \pi_{H_2}(x^{-1}x') )\leq K d(1_G, \pi_{H_1}(x^{-1}x') ), \qquad \forall x,x'\in \Gamma _ {\psi} .
\end{equation}
This last definition has several equivalent expressions when $H_1$ is normal.  
We point out that if a metric Lie group is a semidirect product $G=N  \rtimes H$, then on $G/N$ there is a natural metric that makes the quotient $\pi : G\to G/N$ a submetry, see  \cite[Corollary~2.11]{MR3460162}. However, under the identification $G/N\simeq H$,  this natural distance is biLipschitz equivalent to the one of $G$ restricted to $H$ exactly when the projection on $H$ is Lipschitz, see Proposition~\ref{propEQUIVprop}.

  \begin{prop}\label{ideaCarnot} 
  Let $G=N  \rtimes H$ be a metric group that is a semidirect product.
  
    \noindent{\rm \bf(\ref{ideaCarnot}.i)} If $G$ is a  Carnot group with $N,H$   homogeneous subgroups, then 
 
   \noindent{\rm \bf(\ref{ideaCarnot}.i.a)}
 the map $\pi_H:G\to H$ is a Lipschitz homomorphism
  and
  
  \noindent{\rm \bf(\ref{ideaCarnot}.i.b)} 
  the map $\pi_N:G\to N$ is Lipschitz at $1_G$.
  
  In general, if    {\rm \bf(\ref{ideaCarnot}.i.a)} or   {\rm \bf(\ref{ideaCarnot}.i.b)} holds then the following three properties hold:
  
  \noindent{\rm \bf(\ref{ideaCarnot}.ii)} for all $g\in G$ the set $gN$ is an intrinsically Lipschitz graph.
  
    \noindent{\rm \bf(\ref{ideaCarnot}.iii)}
       If $\psi :N\to H$ is intrinsically Lipschitz map in the FSSC sense, then $\phi : G/H \to G$ defined  as
  \begin{equation}\label{FSSC_Noi}
\phi(gH):=\pi_N(g)   \psi (\pi_N(g)), \quad \forall g \in G
\end{equation}
is an intrinsically Lipschitz section of the projection $\pi :G \to G/H$, with $\phi (G/H)=\Gamma _ {\psi}$. 

    \noindent{\rm \bf(\ref{ideaCarnot}.iv)}
Vice versa, if $\phi :G/H \to G$ is an intrinsically Lipschitz section of  $\pi :G \to G/H,$ then the map $\psi :N \to H$ defined as 
  \begin{equation}\label{Noi_FSSC}
\psi (n):=n^{-1}\phi (nH), \quad \forall n\in N
\end{equation} is an intrinsically Lipschitz map in the FSSC sense, with $\phi (G/H)=\Gamma _ {\psi}$.
  
\end{prop}

The rest of the paper is organized as follows. In $\mathbf{Section \, \ref{sec2}}$ we discuss the definition of intrinsically Lipschitz sections, we show some basic properties like   their continuity, we prove
Proposition~\ref{linkintrinsicocneelip}, and finally we show that in the case when the metric space $X$ is geodesic and the fibers of the projection $\pi$ are one-dimensional and   continuously oriented, the infima of each family of intrinsically Lipschitz sections is so too (see Proposition~\ref{propINFIMUM}).
$\mathbf{Section \, \ref{sec3}}$ contains the proof of Ascoli-Arzel\`a compactness theorem, Theorem~\ref{thm1}. 
$\mathbf{Section \, \ref{sec4}}$ is dedicated to Ahlfors regularity, i.e., the proof of Theorem~\ref{thm2}. 
$\mathbf{Section \, \ref{sec5}}$  contains  the proof of the Extension Theorem (Theorem~\ref{thm3}) using the equivalence between intrinsically Lipschitz sections and level sets of Lipschitz maps that are biLipschitz on fibers. 
$\mathbf{Section \, \ref{sec6}}$ is specialized to the applications of this theory when the metric space $X$ is a Carnot group or, more generally, a metric group.

\textbf{Acknowledgments}. 
A part of this work was done while the authors were at the University of Jyv\"askyl\"a. The excellent work conditions are acknowledged.
During the writing of this work	the authors were partially supported by 
the Academy of Finland 
 (grant 322898
`\emph{Sub-Riemannian Geometry via Metric-geometry and Lie-group Theory}').
E.L.D was also partially supported by
 the Swiss National Science Foundation
(grant 200021-204501 `\emph{Regularity of sub-Riemannian geodesics and applications}')
and by the European Research Council  (ERC Starting Grant 713998 GeoMeG `\emph{Geometry of Metric Groups}').  
The authors thank Andrea Merlo and Danka Lucic for reading an earlier version of this paper.
 
\section{Intrinsically  Lipschitz sections} 
\label{sec2}
\subsection{Preliminaries}
\label{sec:prel}
In this paper $X$ will denote a metric space, whose distance will be denote arbitrarily by $d$, or $d_X$ if there might be confusion with other distances. Instead, the set $Y$ will sometimes be a topological space, and some other times will be a metric space with topology induced by the distance.

As common in topology, a map  
 $\pi : X \to Y$ is called a {\em  quotient map} if it is continuous, surjective and open.
 Distinguished examples of quotient maps are Lipschitz quotients and in particular  submetries, whose definition now we recall. Such notions have been introduced in \cite{MR1736929, Berestovski}.


A map $\pi :X \to Y$ between metric spaces is said to  be a {\em Lipschitz quotient with constant $k$}, with $k\geq 1$ (or briefly a $k$-{\em Lipschitz quotient} or {\em Lipschitz quotient}, if there is no need to specify $k$)
 if
\begin{equation}\label{lipquotient}
B_{d_Y} (\pi (x), r /k) \subset \pi (B_{d_X} (x,r)) \subset B_{d_Y} (\pi(x), kr), \qquad \forall x\in X, \forall r>0.
\end{equation}
If $k=1$, the map $\pi$ is called {\em submetry} and \eqref{lipquotient} simplifies as
\begin{equation}\label{submetryquotient}
 \pi (B_{d_X} (x,r)) = B_{d_Y} (\pi(x), r), \quad \forall x\in X, \forall r>0.
\end{equation}

We stress that being a Lipschitz quotient is more restrictive that being a  quotient map that is Lipschitz. In fact, 
 \eqref{lipquotient} also gives a co-Lipschitz condition. Hence, Lipschitz quotients are uniformly open. In next remark we show that every  quotient map has some type of uniform openness.
 
 \begin{rem}\label{remnewok}
Let $\pi :X \to Y$ be an open map, $K\subset X$ be a compact set and $y\in Y.$ Then $\pi$ is uniformly open on $K \cap \pi ^{-1} (y),$ in the sense that, for every $\varepsilon >0$ there is a neighborhood $U_\varepsilon $ of $y$ such that
\begin{equation*}
U_\varepsilon \subset \pi (B(x,\varepsilon )), \quad \forall x \in K \cap \pi ^{-1} (y).
\end{equation*}
Indeed, since $\pi $ is open, for every $x\in \pi ^{-1} (y)$ there is a neighborhood $U_{\varepsilon , x} $ of $y$ that is contained in $\pi (B( x, \frac  \varepsilon 2) ).$ Moreover, because $K$ is compact, we know that there is a finite $\frac \varepsilon 2$-net $N\subset K \cap \pi ^{-1} (y).$ Finally, if we put $U_\varepsilon := \bigcap _{ x \in N} U_{\varepsilon , x},$ we have that for all $x\in K \cap \pi ^{-1} (y)$ there is a point $\bar x \in N$ such that $d(x,\bar x) < \frac \varepsilon 2$ and
\begin{equation}
U_\varepsilon \subseteq  U_{\varepsilon , \bar x} \subseteq \pi (B( \bar x,  \varepsilon /2) ) \subseteq \pi (B( x,  \varepsilon ) ),
\end{equation}
as wished. 
  \end{rem}

\subsection{Equivalent definitions for intrinsically  Lipschitz  sections}
\label{sec:equiv_def}

\begin{defi}[Intrinsic Lipschitz section]\label{Intrinsic Lipschitz section}
Let $X=(X,d)$ be a metric space and let $Y$ be a topological space. We say that a map $\phi :Y \to X$ is a {\em section} of a quotient map $\pi :X \to Y$ if
\begin{equation*}
\pi \circ \phi =\mbox{id}_Y.
\end{equation*}
Moreover, we say that $\phi$ is an {\em intrinsically Lipschitz section} with constant $L\ $ if in addition
\begin{equation*}
d(\phi (y_1), \phi (y_2)) \leq L d(\phi (y_1), \pi ^{-1} (y_2)), \quad \mbox{for all } y_1, y_2 \in Y.
\end{equation*}
Necessarily, $L\geq 1$.
Equivalently, we are requesting that  that
\begin{equation*}
d(x_1, x_2) \leq L d(x_1, \pi ^{-1} (\pi (x_2))), \quad \mbox{for all } x_1,x_2 \in \phi (Y) .
\end{equation*}
\end{defi}

We further rephrase the definition as saying that $\phi(Y)$, which we call the {\em graph} of $\phi$, avoids some particular sets (which depend on $L$ and $\phi$ itself):


\begin{prop}\label{propo_ovvia} Let $\pi :X \to Y$  be a  quotient map between a metric space and a topological space, $\phi: Y\to X$ be a section of $\pi$, and $L\geq 1$.
Then $\phi$ is $L$-intrinsically Lipschitz if and only if
\begin{equation*}
 \phi  (Y) \cap  R_{x,L} = \emptyset , \quad \mbox{for all } x \in \phi (Y),
\end{equation*}
where $$R_{x,L} := \left\{ x'\in X \;|\;   L d(x', \pi ^{-1} (\pi (x))) <   d(x', x)\right\}.$$


%
\end{prop}

Proposition~\ref{propo_ovvia} is a triviality, still its purpose is to stress the analogy with the FSSC theory. Indeed, the sets $R_{x,L}$ are the intrinsic cones considered in Carnot groups, see Section~\ref{sec:various_notions}.

The case when $\pi :X \to Y$ is a Lipschitz quotient should be considered as the trivial case of our study. Indeed, Condition \eqref{lipquotient} implies that 
$$\frac{1}{k}d(\pi(x_1),  \pi(x_2))\leq d(x_1, \pi ^{-1} (\pi(x_2))) \leq k d(\pi(x_1),  \pi(x_2)),\qquad \forall x_1,x_2\in X.$$ Hence, being intrinsically Lipschitz is equivalent as being a biLipschitz embedding:
\begin{equation*}
\hat L^{-1} \,d( y_1,  y_2 )\leq
d(\phi (y_1), \phi (y_2)) \leq \hat L \,d( y_1,  y_2 ), \quad \mbox{for all } y_1, y_2 \in Y.
\end{equation*}
 We formally state this easy proposition for the record:

\begin{prop}\label{tivial_Lip_quo} Let $\pi :X \to Y$  be a  quotient map between a metric space and a topological space. If one can metrize $Y$ in such a way that $\pi :X \to Y$ becomes a Lipschitz quotient, then a section $\phi: Y\to X$
  of $\pi$ is intrinsically Lipschitz if and only if it is a biLipschitz embedding.
\end{prop}


\begin{exa}\label{baby_example}
The reader could keep in mind the classical fundamental example: For $n,m\in \N$ one considers the projection map $\pi: \R^{n+m}\to \R^n$ on the first $n$ variables, so that  every map $f:\R^n\to \R^m$ has a graphing map $x\in \R^n\mapsto (x,f(x))\in\R^{n+m}$ that is a section of $\pi$. Moreover, such a section is intrinsically Lipschitz (in the sense of  Definition~\ref{Intrinsic Lipschitz section}) if and only if $f$ is Lipschitz in the classical sense.
\end{exa}

\subsection{Intrinsic   Lipschitz   with respect to  families of sections}
In this section we continue to fix a  quotient map $\pi :X \to Y$ between a metric space $X$ and a topological space $Y$.
%

\begin{defi}[Intrinsic Lipschitz  with respect to  a section]\label{defwrtpsinew}
 Given  sections 
  $\phi, \psi :Y\to X$   
  of $\pi$. We say that   $\phi $ is {\em intrinsically $L$-Lipschitz with respect to  $\psi$ at point $\hat x$}, with $L\geq1$ and $\hat x\in X$, if
\begin{enumerate}
\item $\hat x\in \psi(Y)\cap \phi (Y);$
\item $\phi  (Y) \cap  C_{\hat x,L}^{\psi} = \emptyset ,$
\end{enumerate}
where
$$ C_{\hat x,L}^{\psi} := \{x\in X \,:\,  d(x, \psi (\pi (x))) > L d(\hat x, \psi (\pi (x)))  \}.   $$
\end{defi}

  \begin{rem} Definition~\ref{defwrtpsinew} can be rephrased as follows.
 A section $\phi  $ is intrinsically $L$-Lipschitz with respect to  $\psi$ at point $\hat x$ if
 and only if 
 there is $\hat y\in Y$ such that  $\hat x= \phi (\hat y)=\psi(\hat y)$ and
\begin{equation}\label{defintrlipnuova}
 d(x, \psi (\pi (x))) \leq L d(\hat x, \psi (\pi (\hat x))), \quad \forall x \in \phi (Y), 
\end{equation}
which equivalently means 
\begin{equation}\label{equation28.0}
 d(\phi (y), \psi (y)) \leq L d(\psi(\hat y), \psi (y)) ,\qquad \forall y\in Y. 
\end{equation}
  \end{rem}  

  \begin{rem} We stress that Definition~\ref{defwrtpsinew} does not induce  an equivalence relation, because of lack of symmetry in the right-hand side of \eqref{equation28.0}. Still, obviously every section is intrinsically  Lipschitz with respect to  itself. 
  \end{rem}   

The proof of Proposition~\ref{linkintrinsicocneelip} is an immediately consequence of the following result.

 \begin{prop}
   Let $X  $ be a metric space, $Y$ a topological space, and $\pi :X \to Y$   a  quotient map. Let $L\geq1$ and $y_0\in Y$. Assume $\phi_0:Y\to X$ is an intrinsically $L$-Lipschitz section of $\pi$. Let $\phi :Y\to X$ be a section of $\pi$
   such that 
   $x_0:=\phi (y_0)=\phi _0(y_0).$
   Then the following are equivalent:
   \begin{enumerate}
\item  For some $L_1\geq1$, $\phi $ is intrinsically $L_1$-Lipschitz with respect to  $\phi_0$ at   $x_0;$
\item  For some $L_2\geq1$, $\phi $ satisfies   
\begin{equation}\label{equation3nov2021}
d(x_0,\phi (y)) \leq L_2d(x_0, \pi^{-1} (y)), \quad \forall y\in Y.
\end{equation}

\end{enumerate}
   Moreover, the constants $L_1$ and $L_2 $ are quantitatively related in terms of $L$.
   \end{prop}

    \begin{proof}  

[$(1) \Rightarrow (2)$] For every $y\in Y,$ it follows that
     \begin{equation*}
\begin{aligned}
d(\phi (y), x_0) & \leq d(\phi (y), \phi _0(y))  + d(\phi _0(y),x_0)  \\
& \leq (L_1+1)d(\phi _0(y), x_0)  \\
& \leq L(L_1+1)d(x_0, \pi^{-1} (y)),  \\
\end{aligned}
\end{equation*}
   where in the first inequality we used the triangle inequality, and in the second one the intrinsically Lipschitz property of $\phi$ with respect to   $\phi_0$ at   $x_0$. Then, in the third inequality we used the intrinsically Lipschitz property of $\phi_0.$ 
      
      [$(2) \Rightarrow (1)$]   For every $y\in Y,$ we have that
\begin{equation*}
\begin{aligned}
d(\phi (y), \phi _0(y)) & \leq d(\phi (y),x_0) +  d( x_0, \phi _0 (y)) \\
& \leq (L_2+1)  d(  \phi _0 (y)  ,  x_0),
\end{aligned}
\end{equation*}
  where in the first equality we used the triangle inequality, and in the second one we used \eqref{equation3nov2021} and that $\phi _0(y) \in \pi ^{-1} (y)$.  
   \end{proof}

\subsection{Continuity}\label{sec_cont}

An intrinsically $L$-Lipschitz  section $\phi :Y \to X$ of $\pi$ is a continuous map. Indeed, fix a point $y\in Y$ and let $x :=\phi (y)\in X$. Since $\pi$ is open at $x$, for every $\varepsilon >0$   we know that there is an open neighborhood $U_\varepsilon $ of $\pi (x)=y$ such that
\begin{equation*}
U_\varepsilon \subset \pi (B(x,\varepsilon /L)).
\end{equation*}
Hence, if $y'\in U _\varepsilon $ then there is $x'\in B(x,\varepsilon /L)$ such that $\pi(x')=y'$. That means $x'\in \pi^{-1} (y')$ and, consequently,
\begin{equation*}
d(\phi (y), \phi (y')) \leq L d(\phi (y), \pi ^{-1} (y')) \leq Ld(x,x') \leq \varepsilon , 
\end{equation*} 
i.e., $\phi (U_\varepsilon) \subset B(x,\varepsilon ).$

 \begin{exa}\label{rempiopen}
We underline that the fact that $\pi$ is open is a fundamental property in order to obtain the continuity of $\phi.$ Indeed, if we consider $X=Y=\R,$ $A\subset \R$ and $f:A\subset \R \to \R$ be a non-necessarily continuous function with graph $\Gamma _f $. Then the function $\pi: \Gamma _f \to A$ defined as $$\pi (a,f(a))=a,\, \quad a\in A$$ may not be open but the function $\phi :A \to \Gamma _f$ given by $\phi (a)= (a,f(a))$ for $a\in A,$ is an intrinsically Lipschitz section of $\pi.$ On the other hand, it is easy to see that $\pi$ is open if and only if $f$ is a continuous map.
\end{exa}

   \subsection{Infima of intrinsically Lipschitz maps}
 In the case when the metric space $X$ is geodesic and the fibers of the projection $\pi$ are one-dimensional and are continuously oriented, we could consider infima of a family of sections. Possibly, we need to deal with the possibility of values equal to $-\infty$. In next result we prove that if we have a family of intrinsically uniformly  Lipschitz sections, then the infimum is an intrinsically    Lipschitz section, with the possibility of a different value of the intrinsically Lipschitz constant. This latter fact is in accord with Franchi-Serapioni result \cite[Proposition 4.0.8]{FS16}.
    
   \begin{prop}\label{propINFIMUM}
    Let $X$ be a metric  space, $Y$   a topological space, and $\pi:X\to Y$ a quotient map.
    Assume that $X$ is a geodesic space, that there exists a continuous map $\tau:X\to\R$ that is a homeomorphism on the fibers of $\pi$, and  that for each $y\in Y$ the set $ \tau|_{\pi^{-1}(y)}^{-1} (-\infty, 0    )$ is  boundedly compact.
     Let $k\geq 1$, $J$  a set, and for $j \in J$ let $\phi_j :Y \to X$ be an intrinsically $k$-Lipschitz sections. Then 
     either the function $$y\in Y \mapsto \inf\{ \tau(\phi_j(y) ): j\in J\} \in \{-\infty\}\cup\R$$ is constantly equal to $-\infty$
 or    
      the map $\phi :Y \to X$  defined as
    \begin{equation*}
\phi (y):=   \tau|_{\pi^{-1}(y)}^{-1} (\inf\{ \tau(\phi_j(y) ): j\in J\}  )
\end{equation*}
 is well defined on all of $Y$ and it is an intrinsically $k$-Lipschitz section.  
    \end{prop}
     
       \begin{proof}  For each $y$ we define $h(y):=\inf\{ \tau(\phi_j(y) ): j\in J\} \in [-\infty, \infty).$
  Assume that there exists $y_0\in Y$ such that $h(y_0) \ne -\infty$.
  We shall prove some bounds that will also imply that $h(y) \ne -\infty$, for all $y\in Y$.
   Let $y_1,y_2 \in Y.$ 
 For the moment, let us assume that   $h(y_1), h(y_2) \ne -\infty$ so that 
 $\phi(y_i)=\tau|_{\pi^{-1}(y_i)}^{-1}(h(y_i))$, $i=1,2$, is defined as a point in $X$.
    By the definition of infimum, for all $\varepsilon >0$  there is $j_i\in J$ such that   $h(y_i)\leq  \tau(\phi_{j_i}(y_i)) \leq h(y_i)+\varepsilon $, with $i=1,2$, and since $  \tau|_{\pi^{-1}(y_i)}^{-1}$ is continuous, we can also assume that 
      \begin{equation}\label{4nov1532} d( \phi_{j_i}(y_i), \phi(y_i))<\eps.\end{equation}  
 Fix $\eps>0$ and set    $x_i:= \phi_{j_i}(y_i)$.

We want to prove that \begin{equation}\label{non_so}d( x_1, x_2)\leq
k'd( x_2, \pi ^{-1}(y_1)) .\end{equation}
  We consider  $ \bar x_1 \in \pi ^{-1}(\pi(x_1))$ such that $d(x_2, \pi ^{-1}(\pi(x_1))) = d(x_2,\bar x_1)$.
  Let $\gamma$ be a geodesic between $x_2$ and $\bar x_1$.
  Without loss of generality we assume that 
$$\tau(\phi_{j_1}(y_1)) \leq \tau(\phi_{j_2}(y_1))\qquad \text{ and  }\qquad\tau(\phi_{j_2}(y_2)) \leq \tau(\phi_{j_1}(y_2)).$$ Hence,  on the curve $\pi(\gamma)$ there is a point $y^*$ such that $$\tau(\phi_{j_1}(y^*)) = \tau(\phi_{j_2}(y^*))\qquad \text{ and hence  } z^*:=\phi_{j_1}(y^*)) =  \phi_{j_2}(y^*) . $$
Be aware that $z^*$ may not be along $\gamma$, let $z$ be a point on $\gamma$ that is mapped via $\pi$ to $y^*$.
  
  Then we use the triangle inequality with $z^*$, the intrinsically Lipschitz property of $\phi_{j_1}$ (since both $z^*$ and $x_1$ are in its graph),
   the triangle inequality with $x_2$, the intrinsically Lipschitz property of $\phi_{j_2}$ (since both $z^*$ and $x_2$ are in its graph), that $z$ is along $\gamma$,
  we obtain
\begin{equation*} 
\begin{aligned}
d( x_1, x_2)& \leq d( x_1,z^*)+ d(z^*, x_2)\leq
k d( \bar x_1,z^*)+ d(z^*, x_2)\\
& \leq
k (d( \bar x_1,x_2)+ d(x_2,z^*))+d(z^*,x_2)\\ & \leq kd( \bar x_1,x_2)+ k(k+1)d( z,x_2) \leq (2k+k^2) d( \bar x_1,x_2) .
\end{aligned}
\end{equation*}
Thus we proved \eqref{non_so}.

Finally, putting together  \eqref{non_so} and  \eqref{4nov1532} and letting $\varepsilon \to 0$ we get that
\begin{equation}\label{non_so1}d( \phi(y_1), \phi(y_2))\leq
kd( \phi(y_2), \pi ^{-1}(y_1)) .\end{equation}

Now we shall discuss why in the case of the existence of 
 $y_0\in Y$ such that $h(y_0) \ne -\infty$, then we have that the map $\phi$ is well posed, i.e. $h(y) \ne -\infty$, for all $y\in Y$.
The reason is that the same calculation that lead us to \eqref{non_so1} with $y_1=y_0$ and $y_2=y$ arbitrary will give that the values 
$\{\phi_j(y):j\in J\}$ leave in a bounded subset of the fiber $\pi^{-1}(y)$. Because fibers are assumed to be boundedly compact, we have that 
$\{\tau\phi_j(y):j\in J\}$ is a bounded subset of $\R$, which therefore admits a finite minimum.
     \end{proof}

\subsection{The induced distance}\label{sec_induced}
\begin{defi}
 Let $X$ be a metric  space, $Y$   a topological space, and $\pi:X\to Y$ a quotient map. We define the function $d_{\mathcal {F}}: X \times X \to \R^+$  as
\begin{equation}\label{defidf0}
d_{\mathcal {F}}(g_1, g_2):= \frac 1 2 \left( d(g_1,  \mathcal {F}_{g_2} ) + d(g_2,  \mathcal {F}_{g_1} ) \right), \quad \mbox{for all } g_1, g_2 \in X,
\end{equation}
where $\mathcal {F}_{g_i} := \pi ^{-1} (\pi (g_i))$ for $i=1,2$ and $d(g_1,  \mathcal {F}_{g_2} ) := \inf \{ d(g_1, p) \, :\, p\in  \mathcal {F}_{g_2} \}.$
\end{defi}
In general, the map $d_{\mathcal {F}}$ satisfy the following properties:
\begin{itemize}
		\item[(i)] $d_{\mathcal {F}}$ is symmetric, by construction;
		\item[(ii)]  $d_{\mathcal {F}} (g_1, g_2) =0$ if and only if $  \mathcal {F}_{g_1}=\mathcal {F}_{g_2}$;
		\item[(iii)]  $d_{\mathcal {F}} $ does not necessarily satisfies the triangle inequality (see Proposition~\ref{pseudodistance});
\item[(iv)] if we restrict   $d_{\mathcal {F}} $ to a subset of the form $\phi (Y)$ with $\phi $ a section of $\pi$ as in Definition~\ref{Intrinsic Lipschitz section}, then every two points of $\phi (Y)$ have positive distance. 
			\end{itemize}

In (ii), we used that each leave $\mathcal {F}_{g}$ is  a closed set; indeed, $d_{\mathcal {F}} (g_1, g_2) =0$ if and only if $ d(g_1,  \mathcal {F}_{g_2} ) = d(g_2,  \mathcal {F}_{g_1} ) =0$ which is equivalent to say that $g_1$ and $g_2$ belong to the same leaf of $X.$

Notice that
\begin{enumerate}
\item if $\pi : X\to Y$ be a $k$-Lipschitz quotient, then $d(g_1, \mathcal {F}_{g_2}) \leq k d( \mathcal {F}_{g_1},  \mathcal {F}_{g_2});$
\item if $\pi : X\to Y$ be a submetry, then $d(g_1, \mathcal {F}_{g_2}) = d( \mathcal {F}_{g_1},  \mathcal {F}_{g_2}).$
\end{enumerate}

 \begin{prop}\label{pseudodistance}
 Let $X$ be a metric  space, $Y$   a topological space, and $\pi:X\to Y$ a quotient map. If $\phi :Y \to X$ is an intrinsically $L$-Lipschitz section of $\pi$ with $L\geq1,$ then
 \begin{itemize}
		\item[(i)]  when restricted to $\phi (Y)$, the functions $d$ and $d_{\mathcal {F}} $ are $L$-biLipschitz  equivalent; more precisely, it holds
		\begin{equation}\label{equiBIlip}
d_{\mathcal {F}} (p_1,p_2) \leq d(p_1,p_2) \leq L d_{\mathcal {F}} (p_1,p_2),  \quad \forall p_1,p_2 \in \phi (Y).
\end{equation}
		\item[(ii)] $d_{\mathcal {F}}$ when restricted to $\phi (Y) $ is a pseudo distance satisfying the weaker triangle inequality up to multiplication by $L;$ 
		\item[(iii)] it holds
		\begin{equation}\label{inclusionepalle}
\pi \left(B\left(p,\frac r L\right) \right) \subset \pi ( B(p,r) \cap \phi (Y)) \subset \pi (B(p,r)), \quad \forall p\in \phi (Y), \forall r>0.
\end{equation}
			\end{itemize}
  \end{prop}

  \begin{proof}
(i). The left inequality in \eqref{equiBIlip} follows from the simple fact that $p_i \in \mathcal {F}_{p_i}$ and so $d(p_1,  \mathcal {F}_{p_2} ) \leq d(p_1,p_2).$ Regarding the right one, since $\phi$ is intrinsically Lipschitz, we have that 
\begin{equation*}
 d(p_1,p_2) \leq L d(p_1,  \mathcal {F}_{p_2} ) \quad \mbox{ and }  \quad d(p_1,p_2) \leq L d(p_2,  \mathcal {F}_{p_1} ),
\end{equation*}
and, consequently,
\begin{equation*}
 d(p_1,p_2) \leq  \frac 1 2 L(d(p_1,  \mathcal {F}_{p_2} )+d(p_2,  \mathcal {F}_{p_1} )) =Ld_{\mathcal {F}}(p_1,p_2).
\end{equation*}
Hence,  \eqref{equiBIlip} holds.

(ii). We observe that $d_{\mathcal {F}}$ is symmetric, by construction and $d_{\mathcal {F}} (p, p) =0$ because $p\in \mathcal {F}_{p}.$ Moreover, the function $d_{\mathcal {F}}$ satisfies the weaker triangle inequality thanks to (i) and to the fact that  $d$ satisfies the triangle inequality; indeed, we get that
 \begin{equation*}
d_{\mathcal {F}}(p_1,p_2) \leq d(p_1,p_2) \leq  d(p_1,p_3) + d(p_3,p_2) \leq  L (d_{\mathcal {F}}(p_1,p_3) + d_{\mathcal {F}}(p_3,p_2)), 
\end{equation*}
for every $p_1, p_2, p_3 \in \phi (Y).$

(iii). Regarding the first inclusion, fix $p\in \phi (Y), r>0$ and $q\in B(p, \frac r L).$  We need to show that $\pi (q) \in \pi (\phi (Y) \cap B(p,r)).$ Actually, it is enough to prove that 
\begin{equation}\label{equation2.6}
\phi (\pi (q)) \in B(p,r),
\end{equation}
 because if we take $g:= \phi (\pi (q)),$ then $g\in \phi (Y)$ and 
 \begin{equation*}
 \pi (g)= \pi (\phi (\pi (q))) =\pi (q) \in \pi (\phi (Y) \cap B(p,r)).
\end{equation*}
 Hence using the intrinsically Lipschitz property of $\phi$ and the fact that $ \mathcal {F} _q= \mathcal {F} _g$ because $ \pi (g) =\pi (q),$ we have that 
 \begin{equation}
d(p,g) \leq L d(p, \mathcal {F} _g) =L d(p, \mathcal {F} _q) \leq L d(p,q) <L \frac r L =r,
\end{equation}
i.e., \eqref{equation2.6} holds, as desired.  

Finally, the second inclusion in \eqref{inclusionepalle} is trivial, since $\phi (Y)\cap B(p,r) \subset B(p,r)$.
  \end{proof}


\section{An Ascoli-Arzel\`a compactness theorem}\label{sec3}
In this section we finish the proof of Theorem~\ref{thm1}. We already proved  {\rm \bf(\ref{thm1}.ii)} in Section~\ref{sec_cont}. We next restate the missing part.
 \begin{theorem}[Compactness Theorem]
 Let $\pi:X\to Y$ be a quotient map  between a  metric space $X$ for which closed balls are compact and  a topological space $Y$.     Then:
  
\noindent{\rm \bf(i)} For all
    $K' \subset Y$ compact,
  $L\geq 1 $,     $K \subset X$ compact, and  $y_0\in Y$
   the set
 \begin{equation*}\label{set2'}
\mathcal{A}_0:= \{ \phi _{|_{K'}} : K' \to X \,| \, \phi :Y \to X \mbox{ intrinsically $L$-Lipschitz  section of $\pi$}, \phi (y_0) \in K \}
\end{equation*}
is 
 equibounded, equicontinuous, and closed in the uniform convergence topology.
 
 \noindent{\rm \bf(ii)} For all $L\geq 1 $,     $K \subset X$ compact, and  $y_0\in Y$
the set
 \begin{equation*}\label{set1'}
\{ \phi : Y \to X \,: \, \phi \text{ intrinsically $L$-Lipschitz  section of $\pi$},
 \phi (y_0) \in K \}
\end{equation*}
  is compact with respect to  the uniform convergence on compact sets.
  \end{theorem}

\begin{proof}
\noindent{\rm \bf(i).} We shall prove that for all
    $K' \subset Y$ compact,
  $L\geq 1 $,     $K \subset X$ compact, and  $y_0\in Y$
   the set $\mathcal{A}_0$
 is 
\begin{description}
\item[(a)]  equibounded;
\item[(b)] equicontinuous;
\item[(c)] closed.
\end{description}

(a). Fix a compact set $K' \subset Y$ such that  $y_0\in K'.$ We shall prove that for every $y\in K'$ 
 \begin{equation*}\label{set3}
\mathcal{A}:= \{ \phi (y) \,: \, \phi \in \mathcal{A}_0\} 
\end{equation*}
 is relatively compact in $X.$ Fix a point $x_0 \in K$ and let $k:=$ diam$_d(K)$ which is finite because $K$ is compact in $X.$ Then, for every $\phi$ that belongs to $\mathcal{A}_0$, we have that
 \begin{equation*}
\begin{aligned}
d(x_0,\phi (y) ) \leq d(x_0,\phi (y_0) ) + d(\phi (y_0),\phi (y) ) \leq k + Ld(\pi ^{-1}(y), \phi (y_0) ) \leq k +L \max _{ x\in K} d(\pi ^{-1}(y), x ),
\end{aligned}
\end{equation*}
where in the first equality we used the triangle inequality, and in the second one we used the fact that $\phi \in \mathcal{A}_0$ and $x_0\in K$. Finally, in the last inequality we used again $\phi (y_0)\in K$ and that the map $X\ni x \mapsto d(\pi ^{-1}(y) ,x)$ is a continuous map and so admits maximum on compact sets. Since closed balls on $X$  are compact, we infer  that the set $\mathcal{A}$  is relatively compact in $X,$ as desired.

(b). We shall to prove that for every $y\in K'$ and every $\varepsilon >0$ there is an open neighborhood $U_y \subset K' \subset Y$ such that for any $\phi \in \mathcal{A}$  and any $y'\in U_y,$ it follows 
\begin{equation}\label{puntob}
d(\phi (y), \phi( y')) \leq \varepsilon .
\end{equation}

Because of equiboundedness, we have that for every $\phi \in \mathcal{A}_0$ and $ y \in Y$ the set $\phi (y)$ lies within a compact set $K_y$ and so, by Remark~\ref{remnewok}, $\pi$ is  uniformly open on $K_y \cap \pi ^{-1} (y).$ Now let $U_\varepsilon $ an neighborhood of $y$ such that $U_\varepsilon \subset \pi (B(x, \varepsilon /L))$ for every $x\in K_y \cap \pi ^{-1} (y).$ Then we want to show that such  neighborhood $U_\varepsilon $ of $y$ is the set that we are looking for. Take $y'\in U_\varepsilon $ and let $x=\phi (y).$ Hence, there is $x'\in B(x, \varepsilon /L)$ with $\pi (x')=y'$ and, consequently, $x'\in \pi ^{-1}(y').$ Thus we have that for all $\phi$ belongs to $\mathcal{A}_0$
 \begin{equation*}
\begin{aligned}
d(\phi (y),\phi (y') ) \leq L d(\phi (y), \pi ^{-1}(y')) \leq Ld(x,x') \leq L \frac {\varepsilon }{L}  \leq \varepsilon ,
\end{aligned}
\end{equation*} i.e., \eqref{puntob} holds. 
Finally, since the bound is independent on $\phi ,$ we proved the equicontinuity.

(c). By (a) and (b) we can apply Ascoli-Arzel\'a Theorem to the set $\mathcal{A}_0$. Hence, every sequence in it has a converging subsequence. Moreover, this set is closed since if $\phi _h$ is a sequence in it converging pointwise  to $\phi$, then $\phi \in \mathcal{A}_0.$ Indeed, taking the limit of 
\begin{equation*}
d(\phi_h (y), \phi _h (y')) \leq L d(\pi ^{-1} (y), \phi _h (y') ),
\end{equation*}
one gets
\begin{equation*}
d(\phi (y), \phi  (y')) \leq L d(\pi ^{-1} (y), \phi (y') ).
\end{equation*}
Finally, it is trivial that the condition $\phi _h (y_0) \in K$ passes to the limit since $K$ is compact.

\noindent{\rm \bf(ii).} If follows from the latter point $(i)$ using the standard Ascoli-Arzel\'a diagonal argument.
\end{proof}

\section{Proof of  Ahlfors regularity}\label{theoremAhlforsNEW} \label{sec4}

 This section is devoted just to the proof of Theorem  \ref{thm2}. The proof is elementary and only uses the inclusions \eqref{inclusionepalle} from Section~\ref{sec_induced}. Still, we shall see in Section~\ref{sec_Ahlfors_groups} how this new result implies the theorem for intrinsically Lipschitz maps in the FSSC sense.

  \begin{proof}[Proof of Theorem \ref{thm2}]
  Let $\phi$ and $\psi$ intrinsically $L$-Lipschitz sections, with $L\geq1$.
  Fix $y\in Y.$  By Ahlfors regularity of $\phi (Y)$ with respect to $\phi_*\mu$, we know that there are $c_1,c_2, r_0>0$ such that 
     \begin{equation}\label{AhlforsNEW0}
{c_1} r^Q \leq  \phi _* \mu \big( B(\phi (y),r) \cap \phi (Y)\big)  \leq c_2 r^Q,
\end{equation}
for all $0\leq r \leq r_0.$ We would like to show that there are $c_3, c_4>0$ such that
   \begin{equation}\label{AhlforsNEW127}
 {c_3} r^Q \leq  \psi _* \mu \big( B(\psi(y),r) \cap \psi (Y)\big)  \leq c_4 r^Q,
\end{equation}
  for every $0\leq r \leq r_0.$    
  We begin noticing that, by symmetry and  \eqref{Ahlfors27ott.112}
   \begin{equation}\label{Ahlfors27ott}
C^{-1} \mu (\pi (B(\psi(y),r))) \leq \mu (\pi (B(\phi(y),r))) \leq C \mu (\pi (B(\psi(y),r))).
\end{equation}
Moreover, 
\begin{equation}\label{serveperAhlfors27}
  \psi _* \mu \big( B(\psi(y),r) \cap \psi (Y)\big) =  \mu ( \psi^{-1} \big( B(\psi(y),r) \cap \psi (Y)\big) ) = \mu ( \pi \big( B(\psi(y),r) \cap \psi (Y)\big) ), 
\end{equation}
  and, consequently, 
\begin{equation*}\label{ugualecarnot}
\begin{aligned} \psi _* \mu \big( B(\psi(y),r) \cap \psi (Y)\big) & \geq  \mu (\pi (B(\psi (y),r/L))) \geq  C^{-1}   \mu (\pi (B(\phi (y),r/L))) \\
& \geq  C^{-1}   \mu (\pi (B(\phi (y),r/L) \cap \phi (Y))) =    C^{-1} \phi _* \mu \big( B(\phi(y),r/L) \cap \phi (Y)\big) \\ & \geq   c_1C^{-1}  L^{-Q}  r^Q,
\end{aligned}
\end{equation*}
where in the first inequality we used  the first inclusion of \eqref{inclusionepalle}  with $\psi$ in place of $\phi$, and in the second one we used \eqref{Ahlfors27ott}. In the  third  inequality we used the second inclusion of \eqref{inclusionepalle} and in the first equality we used  \eqref{serveperAhlfors27}  with $\phi$ in place of $\psi.$ Moreover, in a similar way we have that 
\begin{equation*}
\begin{aligned} 
\psi _* \mu \big( B(\psi(y),r) \cap \psi (Y)\big) & \leq  \mu (\pi (B(\psi (y),r))) \leq C  \mu (\pi (B(\phi (y),r)))\\
& \leq C  \mu (\pi (B(\phi (y), Lr) \cap \phi (Y)))= C  \phi _* \mu \big( B(\phi(y),Lr) \cap \phi (Y)\big) \\
& \leq  {c_2} C  L^{Q} r^Q.
\end{aligned}
\end{equation*}
Hence, putting together the last two inequalities we have that \eqref{AhlforsNEW127} holds with ${c_3} = c_1C^{-1} L^{-Q}$ and $c_4 = {c_2} C  L^{Q}.$ 
  \end{proof}

\begin{exa}\label{exaAhlfors1711}
Here is an example where some intrinsically Lipschitz sections gives Ahlfors regular graphs and some don't. One can modify Example~\ref{rempiopen} to obtain more pathological examples.
Let $Y =[0,1]  $ be the unit interval and let $X:=I_0 \cup I_1    \subset \R^2$ with $I_i:=\{( x,i)\,:\, x\in [0,1] \},$ for $i=0,1$. 
Here, $X$ is endowed with the following distance:
on pair of points in $I_0$ we consider the Euclidean distance $d_E$ from the plane $\R^2$,
on pair of points in $I_1$ we consider
  $\sqrt d_E,$ and the distance from a  point in $I_0$ to one in $I_1$ is equal to 1, so the triangle inequality is satisfied.
Let the projection $\pi :X\to Y$ be  $\pi ( x,y):=x$.
Then for $i=0,1$ we consider the sections $\phi_i :Y \to X $ defined as  $ \phi_i (x)   :=(x,i).$
Both these two sections  are   intrinsically 1-Lipschitz.
However,  $\phi_0(Y)$ is $1$-Ahlfors regular and  $\psi_1(Y)$ is $2$-Ahlfors regular.  
The example could easily be modified to also have a connected space $X$. And considering instead of $\sqrt d_E$ any other distance on $I_1$, with diameter 1, one can have that   $\psi_1(Y)=I_1 $ is not Ahlfors regular.
\end{exa}

\section{Level sets and extensions}\label{sec5}

In this section we prove Theorem~\ref{thm3}. We shall both generalize and simplify Vittone's argument from \cite[Theorem~1.5]{Vittone20}. We need to mention that there have been several earlier partial results on extensions of Lipschitz graphs, as for example in \cite{FSSC06},   \cite[Proposition~4.8]{MR3194680}, 
   \cite[Proposition~3.4]{V12}, 
 \cite[Theorem~4.1]{FS16}. 
  Regarding extension theorems in metric spaces, the reader can see \cite{ALPD20} and its references.

  \begin{proof}[Proof of Theorem~\ref{thm3}.i]
  Let $f:X\to Z$ and $z_0\in Z$ as in the assumptions of part (\ref{thm3}.i).
We begin recalling that by assumption for every $y\in Y$ the map $f_{|_{\pi^{-1} (y)}} : \pi^{-1} (y) \to Z$ is a biLipschitz homeomorphism  and so it is surjective. Namely, for every $y\in Y$ there is a unique $x \in \pi^{-1} (y)$ such that $f(x)=z_0.$ Hence, it is natural to define $\phi (y):=x$ in such a way \eqref{equationluogoz0_intro} holds trivially. Moreover, we claim that the just-defined section $\phi : Y\to X $ is intrinsically $(1+L^2)$-Lipschitz. Indeed,  for each $y_1, y_2\in Y$ we consider the only points  $x_1 \in \pi^{-1} (y_1) \cap f^{-1} (z_0)$ and $x_2 \in \pi^{-1} (y_2) \cap f^{-1} (z_0),$ and then we shall prove \eqref{eq:def:ILS}, with constant $1+L^2$, showing that
\begin{equation}\label{dis27set}
d(x_1,x_2) \leq (1+L^2) d(x_1, \pi ^{-1}( y_2)).
\end{equation}
For each $\eps>0$, let $\bar x_2 \in \pi ^{-1}( y_2)$ such that 
\begin{equation}\label{dis_aggiustiamo}
d(x_1, \bar x_2)\leq d(x_1, \pi ^{-1}( y_2)) + \eps.\end{equation} Then it follows that 
\begin{equation*}
\begin{aligned}
d(x_1,x_2) & \leq d(x_1, \bar x_2)+  d(\bar x_2, x_2)\\
&  \leq  d(x_1,\bar x_2)+ L d(f(\bar x_2),f(x_2)) \\
& =  d(x_1,\bar x_2)+ L d(f(\bar x_2),f(x_1))\\
& \leq (1+L^2) d(x_1, \bar x_2) \stackrel{\eqref{dis_aggiustiamo}}{\leq}   (1+L^2)  ( d(x_1, \pi ^{-1}( y_2)) + \eps),
\end{aligned}
\end{equation*}
where in the first inequality we used the triangle inequality and in the second inequality we used the  co-Lipschitz property of $f$ on the fiber $\pi ^{-1}( y_2)$; in the  equality we used the fact that $f(x_1)=f(x_2)=z_0$ and finally we used  the Lipschitz property of $f$. 
Consequently, by the arbitrariness of $\eps$, we deduce that  \eqref{dis27set} is true and the proof is complete.
\end{proof}

\medskip

  \begin{proof}[Proof of Theorem \ref{thm3}.ii]
    Let $k$, $L$, $\rho$, $\tau$, and $\{\phi_{\tau_0}\}_{\tau_0}$    as in the assumptions of part (\ref{thm3}.ii).
Fix   $x_0 \in X$, for the moment; and consider $\tau_0:= \tau(x_0)$.  
Recall  that  we have $\tau ^{-1} (\tau_0) = \phi_{\tau_0}(Y)$ by assumption.
We also consider   the function  $\delta_{\tau_0}$, as $ \delta _{\tau_0} (x) := \rho (x_0, \phi_{\tau_0}(\pi (x)))$, which is $k$-Lipschitz on the set $\{|\tau-\tau_0| \leq k L \delta _{\tau_0} \}$ and satisfies $\delta_{\tau_0}(x_0)=0$.
Then, for each such a $x_0$, and $\tau_0$,  we consider the function $f_{x_0} :X \to \R$ defined as
\begin{equation}\label{int}
f_{x_0}(x)=\left\{ 
\begin{array}{lcl}
2(\tau (x)-\tau (x_0)) - \alpha \delta _{\tau _0} (x) &   & \mbox{ if } \, \, |\tau (x)-\tau (x_0)| \leq kL  \delta _{\tau _0}(x)  \\
\tau (x)-\tau (x_0) &    & \mbox{ if }\,\, \tau (x)-\tau (x_0) > kL  \delta _{\tau _0}(x)  \\
3( \tau (x)-\tau (x_0)) &    & \mbox{ if }\,\, \tau (x)-\tau (x_0) < -kL  \delta _{\tau _0}(x),  \\
\end{array}
\right.
\end{equation}
where $\alpha := kL.$  
We prove that the continuous $f_{x_0}$ satisfies the following properties:
\begin{description}
\item[$(i)$]  $f_{x_0}$ is $K$-Lipschitz;
\item[$(ii)$]  $f_{x_0}(x_0)=0;$
\item[$(iii)$]  $f_{x_0}$ is $3k$-biLipschitz on fibers, giving the same orientation that $\tau$ does. 
\end{description}
where $K= \max\{ 3k, 2k+\alpha k \}=2k+\alpha k$ because $\alpha >1.$ The property $(i)$ follows using that $\tau , \delta _{\tau _0}$ are both Lipschitz and $X$ is a geodesic space. On the other hand, $(ii)$ is true since  $\delta _{\tau _0}(x_0)=0 $
Finally, for every $y\in Y$ and $x,x' \in \pi ^{-1} (y) $ we have that $\rho (x_0, \phi_{\tau_0}(\pi (x)))=\rho (x_0, \phi_{\tau_0}(\pi (x')))$, i.e., $\delta_{\tau_0}$ is constant on fibers. Thus, the function $f_{x_0}$ is biLipschitz on fibers because $\tau$ is so too, and actually, the biLipschitz constant is 3 times the constant for $\tau$ and $f_{x_0}$ grows on fibers in the same direction that $\tau$ does.  
Hence $(iii)$ holds.

Now that we have the family $\{f_{x_0}\}_{x_0}$, given $\phi:Y'\to X$ intrinsically $L$-Lipschitz section, we consider the map $f:X \to \R$ given by
\begin{equation*}
f(x) := \sup _{x_0 \in \phi (Y')} f_{x_0} (x), \quad \forall x \in X,
\end{equation*}
and we want to prove that it is the map we are looking for. The Lipschitz properties are valid since the function $\delta _{x_0} $ is constant on the fibers, and (iii) holds. Consequently, the only non trivial fact to show is \eqref{equationluogozeri_intro}. 
Fix $\bar x_0 \in \phi (Y')$.
By $(ii)$ we have that $f_{ \bar x_0}( \bar x_0)=0$ and so it is sufficient  to prove that $f_{x_0}(\bar x_0)\leq 0$ for $x_0\in \phi (Y').$ Let $x_0\in \phi (Y').$  Then  using in addition that $\tau$ is $k$-Lipschitz, and that $\phi$ is  intrinsically $L$-Lipschitz, we have
\begin{equation*}
|\tau (\bar x_0) -\tau( x_0)| \leq k d(\bar x_0 , x_0) \leq Lk d(x_0, \pi^{-1} (\pi(\bar x_0)) ) \leq  Lk d(x_0,  \phi_{\tau_0}(\pi (\bar x_0)) ) =\alpha \delta _{\tau _0} (\bar x _0),
\end{equation*}
and so
\begin{equation*}
f_{x_0}(\bar x_0) = 2(\tau (\bar x_0)-\tau (x_0) - \alpha \delta _{\tau _0} (\bar x_0)) \leq0,
\end{equation*}
i.e., \eqref{equationluogozeri_intro} holds.
   \end{proof}

\section{Applications to groups}\label{sec6}

In this section shall apply the theory developed in the previous sections to the case of groups. 
The general setting is a topological group $G$ together with a closed subgroup $H$ of $G$ in such a way that the quotient space  $G/H:=\{gH:g\in G\}$ naturally is a topological space for which the  map $\pi:g\mapsto gH$  is continuous, open, and surjective: it is a quotient map.

A section for the map $\pi: G\to G/H$ is just a map $\phi: G/H \to G$ such that $\phi(gH)\in gH$, since we point out the trivial identity $\pi^{-1}(gH)=gH$. 
To have the notion of intrinsically Lipschitz section we need the group $G$ to be equipped with a distance which we assume left-invariant. We refer to such a $G$ as a {\em metric group}.

The inequality in definition of intrinsically Lipschitz section (see Definition~\ref{def_ILS}) rephrases as 
\begin{equation}\label{def_ILS_GP}
d\left(\phi (g_1H), \phi (g_2H)\right) \leq L\, d\left(\phi (g_1H), g_2H \right), \quad \mbox{for all } g_1, g_2 \in G.
\end{equation}
The concept of sections and intrinsically Lipschitz sections is preserved by left translation:  namely, if $\Sigma\subset G$ is the graph (i.e., the image) of an intrinsically Lipschitz section and $\hat g\in G$, then $\hat g \Sigma$ is the graph   of some (possibly different)  intrinsically Lipschitz section (see Proposition~\ref{intr graph left translation}). As a consequence, as done by Franchi, Serapioni, and Serra Cassano, one could see the intrinsically Lipschitz condition as a condition near the identity element $1_G$ of $G$ when the graph is translated at $1_G$.
In fact, in the special case in which $\phi (H)=1_G$ equation \eqref{def_ILS_GP}, for $g_1=1_G$, becomes
\begin{equation}\label{def_ILS_GP0}
d(1_G, \phi (g H)) \leq L d(1_G, g H), \quad \mbox{for all }   g  \in G.
\end{equation}

  \begin{prop}[Left-invariance of sections]\label{intr graph left translation}
For each $\hat g  \in G$ and section $\phi: G/H \to G$, the
set $\hat g \phi(G/H)$ is the image of the section
  $\phi_{\hat g} : G/H \to G$ defined as
\begin{equation}\label{defintraslfunct}
\phi_{\hat g} (gH):= {\hat g} \phi ({\hat g}^{-1} gH), \quad  \forall  gH \in G/H.
\end{equation}
Moreover,  $\phi_{\hat g}$ is an intrinsically $L$-Lipschitz section, if so is  $\phi$.
 \end{prop}

  \begin{proof} It is clear that $\phi_{\hat g}$ is a section, since, being $\phi$ a section, we have  $\phi ({\hat g}^{-1} gH) \in   {\hat g}^{-1} gH$.
  It is also evident that the image of $\phi_{\hat g}$ is $\hat g \phi(G/H)$.
Hence, we are just left to prove that if $\phi$ satisfies \eqref{def_ILS_GP},
then so does $\phi_{\hat g}$.
  We use the left invariance of the distance and  the intrinsically Lipschitz property of $\phi$ to obtain 
\begin{eqnarray*}
d(\phi_{\hat g}  (g_1H), \phi_{\hat g}  (g_2H)) & = &d( {\hat g}  \phi ({\hat g}^{-1} g_1H), {\hat g}  \phi ({\hat g}^{-1} g_2H))\\
& = &d(    \phi ({\hat g}^{-1} g_1H),   \phi ({\hat g}^{-1} g_2H))\\
& \leq &Ld(  \phi ({\hat g}^{-1} g_1H),   {\hat g}^{-1}g_2H )  \\
& = &Ld(   {\hat g}\phi ({\hat g}^{-1} g_1H),  g_2H )  \\
&= & Ld(\phi_{\hat g}  (g_1H),  g_2H), 
\end{eqnarray*}
for every $g_1, g_2 \in G$, as desired.
  \end{proof}


\subsection{Splitting of groups and semidirect products}\label{Semidirect products}
Next we shall consider setting where the subgroup $H$ of the metric group $G$ splits, or even more particularly, it splits with respect to a normal subgroup.
In these situations we will have an identification of $G/H$ with a subgroup of $G$, which in our opinion it helps in representing points in the quotient space, but confuses the geometric interpretation of intrinsically Lipschitz sections.

In this section $G$ will be a metric group that admits a splitting $G=H_1\cdot H_2$, as explained in the introduction: $H_1$ and $ H_2$ are two closed subgroups of $G$ for which every element $g\in G$ can be written uniquely as $g= h_1h_2$ with $h_1\in H_1$ and $h_2\in  H_2$. We shall denote $h_1$ as $\pi_{H_1}(g)$ and have a map $\pi _{H_1}:G\to H_1$, and similarly with $\pi_{H_2}:G\to H_2$.

A special splitting is given by semidirect-product structures: one of the factor is normal.
Namely, a group $G$ is a semidirect product if it admits a splitting $G=N\cdot H$ with $N$ normal within $G$.
In other words, the group  $G$ is isomorphic to the structure of semidirect product
$ N \rtimes H$ of     two    groups $N $ and $ H$ where $H$ acts on $N$ by automorphisms. 
When $H$ is seen as subgroup, it  acts on $N$ by conjugation\footnote{We shall repeatedly use the following identity: for any $m,n \in N$ and $\ell \in H$
\begin{equation}\label{formula_projection_conjugation_ginG}
\pi_N (m\ell n) = mC_\ell (n).\end{equation}}: $C_h(n) := hnh^{-1} \in N,$ for all $h\in H$ and $n\in N$.

For the sake of shortness, we shall write that $G=H_1\cdot H_2$ is a {\em splitted metric group} if it is a metric group that admits the splitting $G=H_1\cdot H_2$. If moreover the splitting is a semidirect product we write that 
$G=H_1\cdot H_2$ is a {\em semidirect metric group}.

We stress that if one has a splitting $G=H_1\cdot H_2$ then $G$ also admits the splitting $G=H_2\cdot H_1$. However, the projection maps may be different. For this reason, in this paper we fix the convention that 

{\it we always only 
 consider sections of the quotient with respect to the group on the right}
$$h_1h_2\in H_1\cdot H_2 \stackrel{\pi}{\mapsto} h_1H_2 \in G/H_2. $$
Of course, in the case of a splitting,  we have an identification of $G/H_2$ with $H_1$ element wise. However, as we will see soon, this identification has very little algebraic or geometric significance.

\subsection{Lipschitz property at the identity element} 
\label{sec:Lip_at_1}
We shall consider   the setting of  splitted groups $H_1\cdot H_2$ and consider the various notions of intrinsic Lipschitz graphs.
The key property that will makes us develop a theory  in a way that links the various notions studied in the literature with the very general one that we propose is a type of Lipschitz property for the projection map
$\pi_{H_1}: h_1h_2\in H_1\cdot H_2  {\mapsto} h_1$.
The condition is like the Lipschitz property but fixes one of the two considered points to be the identity element $1$ of the group.  
 We recall that, as defined in \eqref{def:splitat1}, we say that 
  $\pi_{H_1} $ is {\em $K$-Lipschitz at 1} if
  $d(1, \pi_{H_1}(g) )\leq K d(1,g)$, for all $g\in G$. Equivalently, this condition requests that
\begin{equation}\label{def:Lip1}
d(1, h_1 )\leq K d(1,h_1 H_2) , \qquad \forall h_1\in H_1.\end{equation}

The Lipschitz property at the identity element may not hold even in Carnot groups with a semidirect product (see next example), unless the subgroups are homogeneous, see   Proposition~\ref{ideaCarnot}.

 \begin{rem}[Non-example] There are splittings $N \rtimes H$ of subRiemannian Carnot groups for which the projection on $H$ is not Lipschitz, not the projection on $N$ is Lipschitz at 1, not even locally. Here is an example:  Let $ \HH ^1$ be the Heisenberg group seen as $\R^3$ with coordinates $x_1, x_2, x_3$; and let $\{X_1:=\partial _{x_1} - \frac {x_2} 2 \partial _{x_3}, X_2:=\partial _{x_2} +\frac {x_1} 2 \partial_{x_3} , X_3:=\partial _{x_3}\}$ be a basis of its Lie algebra  so that the only non-vanishing relation is $[X_1,X_2]= X_3$. This identification of  $\HH^1$ with $\R^3$ is by means of exponential coordinates associated with $(X_1,X_2,X_3)$.
 The dilations on $ \HH ^1$ become $\delta_\lambda (x_1, x_2, x_3) =  (\lambda x_1, \lambda x_2, \lambda^2 x_3)$, the identity element is ${\bf 0} = (0,0,0)$, and the product law is
 $$ (x_1, x_2, x_3)\cdot (x'_1, x'_2, x'_3) = (x_1+x'_1, x_2+x'_2, x_3+x'_3 +\tfrac12(x_1x'_2 - x_2x'_1)).$$
We  consider the following splitting of $\HH^1$: $N:=\{ (0,x_2,x_3 )\,:\, x_2,x_3 \in \R \}$ and $H:=\{ (x_1,0,x_1 )\,:\, x_1 \in \R \}$, so that $N \rtimes H= \HH ^1$.  
We notice that $N$ is a normal subgroup and $H$ is a {\it non}-homogeneous subgroup of $\HH^1$.  Let $d$ the left invariant metric on $\HH^1$ defined as $d((x_1,x_2,x_3), {\bf 0}) := \max \{|x_1|,|x_2|, \sqrt{|x_3|} \}$, see \cite[pp. 352-353]{LeDonne_Li_Rajala} for the proof that this function gives a distance.
If $g=(x_1,0,0) \in  B ({\bf 0},r)$ for some $r>0,$ then we have that $g=(0,0,-x_1)\cdot (x_1,0,x_1)$ with $(0,0,-x_1)\in N$ and $(x_1,0,x_1) \in H$. Moreover, we have
\begin{equation*}
\begin{aligned}
d({\bf 0},\pi_N (g))& = \sqrt{|x_1|},\\
d({\bf 0},g) &= |x_1|,\\
d({\bf 0},\pi_H (g)) &= \max\{|x_1|,\sqrt{|x_1|}\}.
\end{aligned}
\end{equation*}
Consequently there     is no $L>0$ and  $r>0$ such that for all $g \in B (1,r)$ we would have
 \begin{equation*}
d(1,\pi_N (g)) \leq Ld(1,g), \qquad
\text{nor} \qquad d(1,\pi_H (g)) \leq Ld(1,g).
\end{equation*}
 
\end{rem}

Next we show that if one has the Lipschitz property at the identity element then the standard sections are intrinsically Lipschitz. In case of a splitting $G=H_1\cdot H_2$, the inclusion $i: H_1 \hookrightarrow G$ can be seen as a section of $\pi: G \to G/H_2$ identifying $G/H_2$ with $H_1$. Also, after Proposition~\ref{intr graph left translation}, it is useful to recall that 
$H_1$ is the graph of an intrinsically $k$-Lipschitz section if and only if for all (or, equivalently, for some) $g\in G$ the set $gH_1$ is the graph of an intrinsically $k$-Lipschitz section.

\begin{prop}\label{prop7nov1021}
Let $G=H_1\cdot H_2$ be a splitted metric group   and $K\geq1$. Then the following are equivalent:
\begin{enumerate}
\item the inclusion map $i: H_1 \hookrightarrow G$ is an intrinsically $K$-Lipschitz section of $\pi_{H_1};$
\item $\pi_{H_1} $ is $K$-Lipschitz at $1$; 
\item  one has 
 \begin{equation*}\label{equ.020}
 d(1, \pi_{H_1} (g)) \leq  K d (1, gH_2), \quad \forall g \in G.
\end{equation*} 
\end{enumerate}
\end{prop}

\begin{proof} 

Condition (1), see \eqref{def_ILS_GP},  is equivalent to
\begin{equation}\label{def_ILS_GP'}
d(h_1, h'_1) \leq  K d(h_1,h'_1H_2), \quad \forall h_1,h'_1 \in H_1,
\end{equation}
which by left-invariance is equivalent to \eqref{def:Lip1},
which   is equivalent to Condition (2). 

In addition, since $\pi_{H_1} (gH_2) = \pi_{H_1} (g) $, Condition (3) and \eqref{def:Lip1} are also equivalent.
\end{proof} 

In the case $H_1$ is normal, which means we have a semidirect product $G=N \rtimes H$,
then the map $\pi_{H_1} =\pi_N$ is  Lipschitz at $1$ exactly when the other projection $\pi_{H_2} =\pi_H$ is Lipschitz. We stress that this latter map is a group homomorphism since $N$ is normal. In particular, the map $\pi_H$ is Lipschitz if and only if it is Lipschitz at $1$. These equivalences, with few others, are the subject of next proposition.

 \begin{prop}\label{Defi splitting is locally Lipschitz} 
Let $ G=N \rtimes H $ be a semidirect  metric  group. The following  conditions are equivalent:
 \begin{enumerate}
 \item there is $C_1>0$  such that  $\pi _H : N \rtimes H \to H$ is a  $C_1$-Lipschitz map, 
i.e., 
  \begin{equation*}\label{equ1}
d(\pi_H (g),\pi_H (p)) \leq C_1d(g,p), \quad \forall g,p \in G;
\end{equation*} 
 \item there is $C_2>0$  such that 
 \begin{equation*}\label{equ4}
d(1,\pi_H (g)) + d(1,\pi_N (g)) \leq  C_2d(1,g), \quad \forall g \in G;
\end{equation*}
\item there is $C_3>0$   such that $\pi_N$ is $C_3$-Lipschitz at $1$, i.e.,
 \begin{equation*}\label{equ2}
d(1,\pi_N (g)) \leq C_3d(1,g), \quad \forall g \in G;
\end{equation*}
\item there is $C_4>0$ such that 
 \begin{equation*}\label{equ3}
d(1,\pi_H (g)) \leq C_4d(1,g), \quad \forall g \in G;
\end{equation*} 
\item  there is $C_5>0$ such that 
 \begin{equation*}\label{equ.020'}
 d(1, \pi_N (g)) \leq  C_5 d (g^{-1}, H), \quad \forall g \in G;
\end{equation*} 
\item there is $C_6>0$ such that 
 \begin{equation*}\label{equ.020''}
 d(1, \pi_H (g)) \leq  C_6 d (g, N), \quad \forall g \in G;
\end{equation*}
  \item  there is $ C_7>0$ such that   \begin{equation*}
d(1, C_{\pi_H (g)^{-1}} (\pi_N (g) )) \leq   C_7 \mbox{d} (1,g), \quad \forall g \in G; \\ 
\end{equation*}
  \item  there is $ C_8>0$ such that   \begin{equation*}
d(1, C_{\pi_H (g)^{-1}} (\pi_N (g) )) \leq   C_8 d (g,H), \quad \forall g \in G. \\ 
\end{equation*} 
\end{enumerate} 
\end{prop}

\begin{proof} 
The equivalences $(2) \Leftrightarrow (3) \Leftrightarrow (4)$ easily follow  from the bounds:
 \begin{itemize}
\item $d(1,\pi_N (g)) \leq d(1, g) + d( g, \pi_N (g)) = d( 1, g) + d(1,\pi_H (g)),$
\item $d(1,\pi_H (g)) \leq d(1, (\pi_N (g))^{-1}) + d( (\pi_N (g))^{-1} , \pi_H (g)) =  d(1,\pi_N (g))+ d( 1, g) .$
\end{itemize}

[$(1) \Leftrightarrow (4)$] The implication $(1) \Rightarrow (4)$ follows by taking $p=1$.
The implication $(4) \Rightarrow (1)$ follows because 
$\pi_H$ is a homomorphism: 
\begin{equation*}
\begin{aligned}
d(\pi _H( g), \pi _H( p) )&= d (1, \pi _H( g)^{-1} \pi _H( p)) = d (1, \pi _H(g^{-1} p) ) \\
& \leq  C_4 d (1, g^{-1} p)\\
& =  C_4 d ( g, p ).
  \end{aligned}
\end{equation*}

[$(3) \Leftrightarrow (5)$] This follows from Proposition~\ref{prop7nov1021} $(2) \Leftrightarrow (3).$
 
 [$(4) \Leftrightarrow (6)$] The implication $(6) \Rightarrow (4)$ follows immediately   taking $1\in N$. 
The implication $(4) \Rightarrow (6)$ follows observing that $\pi_H(Ng)=\pi_H(g)$.


For the equivalence of   $(7)$, we show that $(2) \Rightarrow (7)$ and $(7) \Rightarrow (4).$  Notice that for any $nh \in N \rtimes H$ 
 \begin{equation*}\label{equ6}
  \begin{aligned}
 d(1, C_{h^{-1}} (n )) \leq 2 d(1,h) + d(1,n)\leq 2C_2 d(1,nh),
  \end{aligned}
\end{equation*}
we obtain the implication $(2) \Rightarrow (7)$. Moreover, the implication $(7) \Rightarrow (4)$ holds because
\begin{equation*}
 \begin{aligned}
d(1,h) & \leq d(1, nhh^{-1} n^{-1} h) \leq d(1,nh) + d(1, C_{h^{-1}} (n )) \leq (1+C_7) d(1,nh),
\end{aligned}
\end{equation*}
where in the second inequality we used the fact that $d(1,C_{h^{-1}} (n^{-1} ))=d(1, C_{h^{-1}} (n )).$ 

Finally, in order to prove the equivalence of   (8), we show that $(8) \Rightarrow (7)$ and $(3) \Rightarrow (8).$ The implication $(8) \Rightarrow (7)$ follows immediately from $d (g, H) \leq d(g,1).$ The implication $(3) \Rightarrow (8) $ follows by taking   $n \in N, h ,\ell  H$, bounding
 \begin{equation*}
 \begin{aligned}
   d(1, C_{h^{-1} } (n)) =  d(1, C_{h^{-1} } (n^{-1}))
    & =  d(1, \pi_N( C_{h ^{-1} }  (n^{-1})  h^{-1} \ell  ))\\
& \stackrel{(3)}{\leq}  C_3  d(1, C_{h^{-1} } (n^{-1})h^{-1} \ell   ) \\ 
& =   C_3 d(1,  {h^{-1} } n^{-1} \ell)=   C_3 d(nh,    \ell),
 \end{aligned}
\end{equation*}
and taking the infimum over $\ell  H$.


Hence, every two points of the proposition are equivalent and the proof is achieved.
\end{proof}

\begin{rem}
Notice that   many implications in the above proposition are valid also when the splitting   is not a semidirect product, e.g., $(1) \Rightarrow (4), (5) \Rightarrow (3), (6) \Rightarrow (4), (8) \Rightarrow (7)$. \end{rem}

   \begin{rem}\label{Defi_usandog} 
   Using the fact that $N$ is normal, we can rewrite, in the equivalent way, the inequalities of  Proposition~\ref{Defi splitting is locally Lipschitz}:
   \begin{itemize}
\item  Observe that $d(g, \pi_N(g))=d(1, \pi_H(g))$, we can change the left term in the inequalities (4) and (6) in Proposition~\ref{Defi splitting is locally Lipschitz} with  $d(g, \pi_N(g)).$
\item Notice that $d(g, \pi_H(g))= d(1,g^{-1} \pi_H(g)) =d(1, C_{\pi_H (g)^{-1}} (\pi_N (g) )),$ we can change the left term in the inequalities (7) and (8) in Proposition~\ref{Defi splitting is locally Lipschitz}  with $d(g, \pi_H(g)).$ 
\item Observe that $C_{ g^{-1}} (\pi_N (g) ) = C_{\pi_H (g)^{-1}} (\pi_N (g) )$, we can change the left term in the inequalities (7) and (8) in Proposition~\ref{Defi splitting is locally Lipschitz}  with $d(1,C_{ g^{-1}} (\pi_N (g) )).$
\end{itemize}

\end{rem}

 \begin{rem}\label{remservepericoniX} In every metric  group
  $ (N \rtimes H,d)$ one has the following inequalities:  
  
 \begin{equation*}
 \begin{aligned}
& d (g, N)  \leq d(1, \pi_H (g)), \\
 & d (g, H)  \leq d(1, C_{\pi_H (g)^{-1}} (\pi_N (g) )),  \\ 
 & d(1, g)  \leq d(1, C_{\pi_H (g)^{-1}} (\pi_N (g)  ))+ d(1, \pi_H (g)),\qquad \forall  g \in N \rtimes H . \\
 \end{aligned}
\end{equation*}
 Indeed,   considering $g \in N \rtimes H$ so that 
  $g=\pi_N (g) \cdot \pi_H (g)$,  we have
 \begin{equation*}
  \begin{aligned}
 d (g, N) & \leq d(1, (\pi_H (g))^{-1}\cdot (\pi_N (g))^{-1}\cdot \pi_N (g))= d(1, \pi_H (g)),\\
 d (g, H) & \leq d(1, (\pi_H (g))^{-1}\cdot (\pi_N (g))^{-1}\cdot \pi_H (g)) = d(1, C_{\pi_H (g)^{-1}} (\pi_N (g) )).
  \end{aligned}
\end{equation*}
    Moreover, to prove the last inequality it is enough to notice that
 \begin{equation*}
  \begin{aligned}
g= \pi_N (g) \cdot \pi_H (g)=\pi_H (g) \cdot  [\pi_H (g)]^{-1} \cdot \pi_N (g) \cdot \pi_H (g) = \pi_H (g) \cdot  C_{\pi_H (g)^{-1}} (\pi_N (g) ).
  \end{aligned}
\end{equation*}  
\end{rem}

%

\subsection{Lipschitz projections for CC distances}
In order to understand why in Carnot groups equipped with a homogeneous splitting the various notions of intrinsically Lipschitz sections coincide, we shall get a criterion to determine when the projection on a factor of a splitting of a group is Lipschitz. In this subsection, we shall focus on Carnot-Carath\'eodory distances on groups, see an introduction in \cite{ledonne_primer} for the notion of CC-metric induced by a distribution $\Delta$.

 \begin{prop}\label{propLip.23} Let $G= N \rtimes H$ be the semidirect product of two  Lie groups.
 Let $\Delta \subset \mathfrak g =  \mathfrak  n \rtimes \mathfrak  h$ be a bracket generating distribution on $G$.
 Then the following statements are equivalent:
  \begin{enumerate}
  \item $\pi_{\mathfrak  h} (\Delta ) \subset \Delta $.
 \item $\pi _H$ is Lipschitz 
 for every CC-metric induced by $\Delta$.
\end{enumerate} 
\end{prop}

In the proposition, we denoted by $\pi_{\mathfrak  h} $ the projection from the Lie algebra $\mathfrak  g$ of $G$ to 
the Lie algebra $\mathfrak  h$ of $H$ to the 
modulo
the Lie algebra   $\mathfrak  n$ of $N$.

Before the proof of Proposition \ref{propLip.23}, we discuss a lemma.
 \begin{lem}\label{propLip.0} Let $ N \rtimes H$ be the semidirect product of two  Lie groups.
Let $k\in \N$ and  $\Delta \subseteq   \mathfrak  n \rtimes \mathfrak  h$ be a 
  k-dimensional linear subspace  of the Lie algebra.
  If $m:=\dim(\pi_{\mathfrak  h} (\Delta ))$ then 
    there are $X^{\mathfrak  h}_1,\ldots , X^{\mathfrak  h}_m \in \mathfrak h $ and $ X^{\mathfrak n}_1, \ldots , X^{\mathfrak n}_k \in \mathfrak n$  such that
   \begin{equation}\label{baseperDelta.20.0}
 X^{\mathfrak  h}_1+X^{\mathfrak n}_1 ,\ldots , X^{\mathfrak  h}_m +X^{\mathfrak n}_m  ,  X^{\mathfrak n}_{m+1}, \ldots , X^{\mathfrak n}_k 
 \quad \mbox{ is a basis for } \Delta.
\end{equation}
Moreover, if   
    \begin{equation}\label{ipotesi_Delta}\pi_{\mathfrak  h} (\Delta ) \subset \Delta ,\end{equation} 
   we may choose 
   $ X^{\mathfrak n}_1= \ldots = X^{\mathfrak n}_m=0,$ so that 
      \begin{equation}\label{baseperDelta.20.1}
 X^{\mathfrak  h}_1  ,\ldots , X^{\mathfrak  h}_m    ,  X^{\mathfrak n}_{m+1}, \ldots , X^{\mathfrak n}_k 
 \quad \mbox{ is a basis for } \Delta.
\end{equation}
   \end{lem}
%

 \begin{proof} 
 Recall that $\pi:=\pi_{\mathfrak  h} :   \mathfrak  n \rtimes \mathfrak  h \to    \mathfrak  h $ is the projection onto $  \mathfrak  h$ modulo $ \mathfrak  n $. We shall consider the restriction of it to $\Delta$, that is $\pi|_\Delta: \Delta \to \pi(\Delta).$ 
 Recall that       \begin{equation}\label{algebra_lineare}k=\dim(\Delta)=\dim ( \pi(\Delta)) + \dim(\ker( \pi|_\Delta))= m + \dim(\ker( \pi|_\Delta)).
 \end{equation}
  Thus
 $\dim(\ker( \pi_\Delta))= k-m$. Hence, let $ X^{\mathfrak n}_{m+1}, \ldots , X^{\mathfrak n}_k \in \mathfrak n$  be a basis of $\ker( \pi_\Delta)$.
Also, let 
 $X^{\mathfrak  h}_1,\ldots , X^{\mathfrak  h}_m $ be a basis of $\pi(\Delta)$.
   In particular, notice that 
   $$ X^{\mathfrak n}_{m+1}, \ldots , X^{\mathfrak n}_k \in \ker( \pi|_\Delta)\subseteq \mathfrak n\cap \Delta$$ 
   and
    $$X^{\mathfrak  h}_1,\ldots , X^{\mathfrak  h}_m \in \pi(\Delta)\subseteq \mathfrak  h.$$
    For each $i=1,\ldots, m$, since $X^{\mathfrak  h}_i \in \pi(\Delta)$ and since $\ker( \pi|_\Delta)\subseteq \mathfrak n$ there exists  $X^{\mathfrak  n}_i \in {\mathfrak n}$ such that 
     $X^{\mathfrak  h}_i +X^{\mathfrak  n}_i \in\Delta$. 
  Therefore, from \eqref{algebra_lineare} we have that \eqref{baseperDelta.20.0} holds true.
   
   If, in addition, we have 
   \eqref{ipotesi_Delta} then we can choose $X^{\mathfrak  n}_i=0$, for all $i=1,\ldots, m$. And we conclude \eqref{baseperDelta.20.1}.
   %
%
%
%
%
  \end{proof}

 \begin{proof}[Proof of Proposition~\ref{propLip.23}] 
 [$(1) \Rightarrow (2)$] Let $\Delta':= \Delta\cap \mathfrak h$. Fix a left-invariant scalar product on $\mathfrak g$. Let $d'$ be the CC-distance on $H$ determined by  $\Delta'$. Notice that since for the definition of $d'$ one only considers $\Delta$-horizontal curves within $H$, we have that \begin{equation} \label{relazione_dist}d' \geq d_H,\end{equation} where $d_H$ denotes the CC  distance $d$ determined by  $\Delta$ on $G$ restricted to $H$.

We notice that $\pi_{\mathfrak  h} (\Delta ) \subset \Delta' $ if and only if the (smooth
homomorphic) map $\pi_H: (G,d) \to (H, d')$ is Lipschitz on compact sets. Since the map is a group morphism and the distance is geodesic, then there is no difference between Lipschitz and locally Lipschitz.
Hence by \eqref{relazione_dist}, these last conditions   imply that 
$\pi_H: (G,d) \to (H, d_H)$ is Lipschitz. 
 
%

[$(2) \Rightarrow (1)$] 
 By contradiction, we assume that $\pi_{\mathfrak  h} (\Delta ) \nsubseteq \Delta ,$ i.e., there is $w \in \Delta$ such that  
 $\pi_{\mathfrak  h} (w)= w_1 \in \pi_{\mathfrak  h} (\Delta )  \setminus \Delta $. Hence, by $\pi_{\mathfrak  h} (\Delta ) \subset \mathfrak  h$, we have that $w_1 \in \mathfrak h \setminus \Delta$. 
 
Now if $w= w_1+w_2$ with $w_2 \in \mathfrak n$, then for some $t>0$ we have that $tw= tw_1+tw_2 \in B(1,r)$ for some $r>0$ and so, using the facts $tw\in \Delta$ and $tw_1 \notin \Delta$. 
 \begin{equation*}
d(1,tw) \sim t, \quad \mbox{and} \quad d(1, t w_1) \gg t.
\end{equation*}
 Now, since $\pi_H: (G,d) \to (H, d_H)$ is assumed $L$-Lipschitz, it follows that $$t \ll d(1, t w_1)  \leq Ld(1,tw) \sim t$$ and so the contradiction.
  \end{proof}

\subsection{Sections in semidirect products and FSSC conditions}

Next we make some links between our notion of intrinsically Lipschitz section and  the various notions of intrinsically Lipschitz maps in the sense of Franchi, Serapioni, and Serra Cassano. The setting we are considering is the case of a splitting $G=H_1\cdot H_2$.
There will be a double view point in the objects of study: On the one hand, we might consider sections $\phi:G/H_2\to G$. On the other hand, we might consider maps $\psi:H_1\to H_2$.
There is an obvious link between the two objects:
A map $\psi$ induces a section $\phi$ as
  \begin{equation*}
\phi(gH_2):=\pi_{H_1}(g)   \psi (\pi_{H_1}(g)), \quad \forall g \in G.
\end{equation*}
A   section $\phi$   induces a map $\psi$ as
$$\psi (n):=n^{-1}\phi (nH_2), \quad \forall n\in {H_1}.$$
For a map $\psi:H_1\to H_2$ we set
  $\Gamma _ {\psi} := \{n\psi(n)\,:\, n \in H_1\},$ which is exactly the image of the associated section. 
  
  We say that $\psi$ is an 
{\em  intrinsically Lipschitz map in the FSSC sense}  if  exists $K>0$ such that
\begin{equation}\label{equation2312}
d(1, \pi_{H_2}(x^{-1}x') )\leq K d(1, \pi_{H_1}(x^{-1}x') ), \qquad \forall x,x'\in \Gamma _ {\psi} .
\end{equation}
In Lemma~\ref{propNEWequivNATALE}, we shall soon see that  this condition is equivalent to require that 
   \begin{equation}\label{6nov848.0..8}
d(1, x^{-1}x' )\leq \tilde K d(1, \pi_{H_1}(x^{-1}x') ), \qquad \forall x,x'\in \Gamma _ {\psi} .
\end{equation} 

The last property that we consider for a section is the following:
Given a splitting $G=H_1\cdot H_2$, we say that a section
for $\phi:G/H_2\to G$ is {\em intrinsically Lipschitz   with respect to the standard sections} if it is intrinsically Lipschitz  with respect to every  section of the form $ gH_1$ for all $g\in\phi(G/H_2)$, see 
Definition~\ref{defwrtpsinew}. Explicitly, a set $\Sigma\subset G$ is the graph of an intrinsically Lipschitz map with respect to the standard sections if and only if
\begin{equation}\label{defintrlipnuova2}
 d(g, gH_2\cap \hat g H_1) \leq L d(\hat g, gH_2\cap \hat g H_1), \quad \forall g,\hat g \in \Sigma.
\end{equation}
When $gH_2\cap \hat g H_1$ is a singleton for all $g,\hat g \in \Sigma$ (this happen for instance when $H_1$ is normal), this condition is more general than those mentioned above; indeed, for any $x,x' \in  \Sigma = \Gamma _ {\psi}$ we have that
 \begin{equation}\label{eqServePOI.2412}
\begin{aligned}
d(x,x') & \leq d(x, xH_2 \cap x' H_1)+ d( xH_2 \cap x' H_1, x')\leq (L+1)  d( x', xH_2 \cap x' H_1)\\
& \leq (L+1) d(x',  x' H_1) \leq  (L+1)  d(1, \pi_{H_1}(x^{-1}x') ).\\
\end{aligned}
\end{equation}
Yet,  when $\pi_{H_1} $ is $k$-Lipschitz at $1$ and $H_1$ is normal,  the condition \eqref{defintrlipnuova2} is equivalent to \eqref{equation2312} and  \eqref{6nov848.0..8} (see Proposition~\ref{propNEWequivNATALE.2}).

\subsection{The trivial case when the quotient map is a Lipschitz quotient}

We shall spend some words remarking that in the case we are in a group on which we are taking the quotient modulo a normal subgroup, then the quotient map is a Lipschitz quotient with respect to a distance on the quotient space. Hence, by Proposition~\ref{tivial_Lip_quo} the theory of intrinsically Lipschitz sections coincides with the one of biLipschitz embeddings. If moreover, the group has a splitting, then our intrinsically Lipschitz sections coincide with the Lipschitz maps between the factors.


 \begin{prop}\label{propEQUIVprop}
  Let $G = N \rtimes H$ be a semidirect metric group. Assume $N$ is boundedly compact so to have a  quotient  metric $d_{G/N}$ on $G/N$ (see \cite[Corollary 2.11]{MR3460162}). Via the projection on $H$ given by the semidirect product we see  $d_{G/N}$ as a distance on $H$. Then, the following facts are equivalent:
 \begin{enumerate}
\item the projection $\pi _H :G \to H$ is $L$-Lipschitz map$;$
\item it holds
\begin{equation}\label{equationBIlip}
 d_{|_H} (h, \ell) \leq L d_{G/N} (h,\ell ),\quad \forall h, \ell \in H .
\end{equation}
\end{enumerate}
Moreover, if one of these conditions are true then $d_{|_H}$ and $d_{G/N}$  are biLipschitz equivalent:
\begin{equation*}\label{equationBIlip.0}
\frac 1 L d_{|_H} (h, \ell) \leq d_{G/N} (h,\ell )\leq  d_{|_H} (h,\ell),\quad \forall h, \ell \in H .
\end{equation*}
 \end{prop}

 \begin{proof} 
 
[$(1) \Rightarrow (2)$] Fix $h, \ell \in H.$ Recall that there are $p,q \in G$ such that $\pi_H (p)=h, \pi_H (q)=\ell$ and $d_{G/N} (h,\ell ) =d(p,q),$  we get that
 \begin{equation*}
\begin{aligned}
 d_{|_H} (h, \ell)  = d(\pi_H(p), \pi_H (q)) \leq L d(p,q) =L d_{G/N} (h,\ell ),
\end{aligned}
\end{equation*}
 where in the first inequality we used the Lipschitz property of $\pi_H.$ 

[$(2) \Rightarrow (1)$] We notice that for every $p,q \in G$ with $\pi_H(p)=h$ and $\pi_H (q)=\ell$
  \begin{equation*}
\begin{aligned}
 d(\pi_H (p), \pi_H(q)) & = d(h,\ell ) \leq  L  d_{G/N} (h,\ell ) = L d(\pi ^{-1}_H (h), \pi ^{-1}_H (\ell)) \leq  L d(p,q).
\end{aligned}
\end{equation*}
The last statement follows from the simple fact that $ d_{G/N} (h,\ell ) = d(\pi ^{-1}_H (h), \pi ^{-1}_H (\ell)) =d (Nh, N\ell ) \leq d(h, \ell).$  
 \end{proof} 

From 
Proposition~\ref{tivial_Lip_quo} we have the following consequence.

\begin{coroll}\label{propEQUIVprop.0}
  Let $G =   N\rtimes H$ semidirect metric group with $N$ boundedly compact. If the projection $\pi _H :G \to H$ is Lipschitz, then every 
 intrinsically Lipschitz section $\psi: N\to G$ for $\pi_N$
  is a Lipschitz embedding.
\end{coroll}

\subsection{Link between the various notions}\label{sec:various_notions}

Recall that in a group that admits a splitting $G=H_1\cdot H_2$, to every map $\psi:H_1\to H_2$ we associate its graph
 $\Gamma _ {\psi}:= \{n\psi(n)\,:\, n \in H_1\}\subset G$.

\begin{lem}\label{propNEWequivNATALE}
Let $G=H_1\cdot H_2$ be a splitted metric group. For every  $\psi:H_1\to H_2$, the following are equivalent:
 \begin{enumerate}
 \item  $\psi$  is an intrinsically $K$-Lipschitz map in the FSSC sense, as in \eqref{equation2312}$;$
\item it holds
 \begin{equation}\label{27gennaio2022}
  d(x, x' ) \leq  \tilde K d(1, \pi_{H_1}(x^{-1}x') ), \qquad \forall x,x'\in \Gamma _ {\psi} ;
\end{equation}
\end{enumerate}
\end{lem}

\begin{proof}

 [$(1) \Rightarrow (2)$] Using the triangle  inequality we have that for any $x,x'\in \Gamma _ {\psi} $
 \begin{equation*}
    d(x, x' ) =d(1, x^{-1}x' ) \leq  d(1, \pi_{H_1}(x^{-1}x') ) + d(1, \pi_{H_2}(x^{-1}x') ) \leq (K+1) d(1, \pi_{H_1}(x^{-1}x') ).
\end{equation*}
[$(2) \Rightarrow (1)$] Using the left invariant property of $d$ and the triangle  inequality we obtain that for any $x,x'\in \Gamma _ {\psi} $
 \begin{equation*}
d(1, \pi_{H_2}(x^{-1}x') )  \leq   d(1, x^{-1}x' ) + d(1, \pi_{H_1}(x^{-1}x') ) \leq (\tilde K+1) d(1, \pi_{H_1}(x^{-1}x') ).
\end{equation*}
\end{proof}

\begin{prop} \label{propNEWequivNATALE.2}
Let $G=N  \rtimes H$ be a semidirect metric group such that $\pi_{N} $ is $k$-Lipschitz at $1$. For every  $\psi:N\to H$, the following are equivalent:
 \begin{enumerate}
 \item  $\psi$  is an intrinsically $K$-Lipschitz map in the FSSC sense as in \eqref{equation2312}$;$
\item $\Gamma _ {\psi} \subset G$ is the graph of an intrinsically Lipschitz map with respect to the standard sections, i.e., \eqref{defintrlipnuova2} holds for every $g,\hat g \in \Gamma _ {\psi}.$
\end{enumerate}
\end{prop}

\begin{proof} 

[$(2) \Rightarrow (1)$] This follows from \eqref{eqServePOI.2412} noticing that the set $xH \cap x' N$ is a singleton for every $x=n\psi(n),x'=m\psi(m) \in \Gamma _ {\psi},$ with $n,m\in N$ Indeed, using the fact that $N$ is normal, if $nh=m\psi(m) n' \in xH \cap x' N,$ for some $h\in H$ and $n\in N$, then
\begin{equation*}
nh= \underbrace{m C_{\psi(m) } ( n')  }_{ \in  N}  \underbrace{\psi (m) }_{ \in  H},
\end{equation*}
and so by uniqueness of the projection on $N$ and on $H$ we get that $h= \psi (m)$ and $n'= C_{\psi(m)^{-1}  } ( m^{-1}n). $

  [$(1) \Rightarrow (2)$] Using Lemma~\ref{propNEWequivNATALE} and recall that the set $xH \cap x' N$ is a singleton, for any $x,x' \in \Gamma _ {\psi}$ we have that
 \begin{equation*}
\begin{aligned}
d(x, xH \cap x' N) & \leq  d(x, x') + d(x', xH \cap x' N)\\
& \leq (K+1)d(1, \pi _{N} ((x')^{-1}x)) + d(x', xH \cap x' N) \\
& \leq C(K+1) d(1, (x')^{-1}x H) + d(x', xH \cap x' N),\\ 
\end{aligned}
\end{equation*} where in the last inequality we used Proposition \eqref{prop7nov1021} (3). 
Now we consider $h_2 \in H$ such that $d(x',x H \cap x'N)  = d(x', xh_2)$ and, consequently,
 \begin{equation*}
\begin{aligned}
d(x, xH \cap x' N) & \leq C(K+1) d(x',x h_2) + d(x', xH \cap x' N) = (C(K+1)+1) d(x', xH \cap x' N). 
\end{aligned}
\end{equation*}

\end{proof}

In the context of metric groups, Proposition~\ref{linkintrinsicocneelip} is as follows. Regarding Carnot groups, the reader can see \cite{FS16, SC16} and their references. 

 \begin{coroll}\label{corol8.8} 
Let $G=H_1\cdot H_2$ be a splitted metric group such that $\pi_{H_1} $ is $k$-Lipschitz at $1$.   Let $\psi:H_1 \to H_2$, $h_1 \in H_1$ and $p =h_1 \psi(h_1).$ Then the following statements are equivalent:
 \begin{enumerate}
 \item $\psi$ is intrinsically $L$-Lipschitz in the FSSC sense  at  $h_1 \in H_1;$ 
\item  for all $\hat L\geq (L+1)k,$ it holds $$p\cdot X_{H_2} (1/\hat L)  \cap \Gamma_{\psi } =  \emptyset,$$
where $p\cdot X_{H_2}( \alpha )$ is the cone with  axis $H_2,$ vertex $p,$ opening  $\alpha $ defined as the translation of 
\begin{equation*} 
X_{H_2}(\alpha) :=\{ g \in G \, :\, d (g^{-1},H_2) < \alpha d(1,g) \}.
\end{equation*}
\end{enumerate} 
\end{coroll}

\begin{proof} 
It is enough to combine Lemma~\ref{propNEWequivNATALE} and Proposition~\ref{prop7nov1021}.
\end{proof}

We conclude this section proving Proposition~\ref{ideaCarnot}.

 \begin{proof}[Proof of Proposition \ref{ideaCarnot}]

   \noindent{\rm \bf(\ref{ideaCarnot}.i.a)} and   \noindent{\rm \bf(\ref{ideaCarnot}.i.b)}. The  statements can be either found in   \cite[Proposition 2.2.9]{FS16}, or, more generally, they follow from 
   Proposition~\ref{propLip.23}
    and Proposition~\ref{Defi splitting is locally Lipschitz}.
   

  \noindent{\rm \bf(\ref{ideaCarnot}.ii)}. Because $N$ is a subgroup and because of left-invariance of intrinsically Lipschitz sections (see Proposition~\ref{intr graph left translation}), it is enough to prove that $N\simeq G/H \hookrightarrow G$ is an intrinsically  Lipschitz section of $\pi_{N}$. 
  From Proposition~\ref{prop7nov1021}, we conclude if we have {\rm \bf(\ref{ideaCarnot}.i.b)} (or, equivalently from Proposition~\ref{Defi splitting is locally Lipschitz}
  if we have {\rm \bf(\ref{ideaCarnot}.i.a)}).


   \noindent{\rm \bf(\ref{ideaCarnot}.iii)} 
     We want to prove \eqref{def_ILS_GP} for $\phi$. 
%
Notice that from the definition  \eqref{FSSC_Noi}  of $\phi$ and the fact that $\psi$ is ranged into $H$, we have
   \begin{equation}\label{6nov8480} \phi (g_2 H) H =\pi_N(g_2)   \psi (\pi_N(g_2)) H= g_2H.
   \end{equation}
   Since $\psi:N\to H$ is intrinsically Lipschitz map in the FSSC sense, by Lemma~\ref{propNEWequivNATALE}, we   have \eqref{27gennaio2022}, once we observe  that $\phi$ is ranged into $\Gamma_\psi$. 
Hence, 
we use   Proposition~\ref{prop7nov1021}.(3) to  get the desired equation \eqref{def_ILS_GP}:  
   \begin{equation*}
\begin{aligned}
d(\phi(g_1 H), \phi (g_2 H)) &  \stackrel{\eqref{27gennaio2022}}{\leq}  \tilde K d(1, \pi_{N}(\phi(g_1 H)^{-1} \phi (g_2 H) ))\\
&\stackrel{\eqref{prop7nov1021}}{\leq}  \tilde K K d(1, \phi(g_1 H)^{-1} \phi (g_2 H)H)  \stackrel{\eqref{6nov8480}}{=}  \tilde K K d(\phi(g_1 H), g_2 H).
\end{aligned}
\end{equation*}


    \noindent{\rm \bf(\ref{ideaCarnot}.iv)} 
     Vice versa, we want to prove   \eqref{27gennaio2022}   for the map $\psi$ defined as 
   in \eqref{Noi_FSSC},  
     assuming \eqref{def_ILS_GP}. 
     First, for all $n,m\in N$ observe that, since $\psi$ is ranged into $H$ we have that
    \begin{equation}\label{proj_gen}
    \pi_{N}((n\psi (n))^{-1} m\psi (m))) =   \pi_{N}((n\psi (n))^{-1} m\psi (n)))=(n\psi (n))^{-1} m\psi (n),
    \end{equation}
     where in the  
      last equality we used  that $N$ is normal.  
Then, for all $n,m\in N$ we have 
   \begin{equation*}
\begin{aligned}
d(n\psi (n),m \psi (m)) & \stackrel{\eqref{Noi_FSSC}}{=}  d(\phi(nH), \phi (mH))  \stackrel{\eqref{def_ILS_GP}}{\leq} L d(\phi (nH), mH)\\
&   \stackrel{\eqref{Noi_FSSC}}{=} L d(1, (n\psi (n))^{-1} m H) \\
& \leq L d(1,  (n\psi (n))^{-1} m\psi (n)) \stackrel{\eqref{proj_gen}}{=} L d(1,  \pi_{N}((n\psi (n))^{-1} m\psi (m)))  .\\
\end{aligned}
\end{equation*}
 where in the  inequality we used that $\psi$ is ranged into $H$. 
  \end{proof}



\subsection{Ahlfors regularity in groups}
\label{sec_Ahlfors_groups}

A corollary of Theorem~\ref{thm2} is Theorem~\ref{theoremAhlfors} which roughly states that an intrinsically Lipschitz graph   on a normal Ahlfors-regular subgroup $N$  is  Ahlfors-regular. 


 \begin{rem}\label{remHaar}
Given a left-Haar measure $\mu_N$ on a closed normal subgroup $N \triangleleft G$, then the measure  
$\mu_N$ may not be preserved by conjugation by elements in $G$.
However, assuming in addition that $G$ is a  Lie group, we claim that for all compact sets $K\subseteq G$ there exists $C>0$ such that for all $g\in K$ we have that on $K$
the Jacobian of $C_g$  with respect to $\mu_N$ is bounded by $C$.
Indeed, this last statement follows from the fact  that on Lie groups every Haar measure is given by a smooth volume form and each map $C_g:N\to N$ is smooth.  
\end{rem}

 At this point we have an easy rephrasing of Theorem~\ref{thm2} in the case of groups. Still, we provide the short proof next.
 
 \begin{theorem}\label{theoremAhlfors} 

Let  $G=N \rtimes H $ be a semidirect  metric Lie group with 
boundedly compact distance. 
Assume that $\pi_H: G \to H$ is Lipschitz and that $N$ is locally $Q$-Ahlfors regular.
If $\phi : G/H\to G$ is an intrinsically $L$-Lipschitz section, then $\phi (G/H)$ is locally  $Q$-Ahlfors regular. 
\end{theorem}
Recall that requiring that $N$ is locally $Q$-Ahlfors regular 
means, first that the $Q$-Hausdorff measure $\mu_N$ of $N$ is locally finite and nonzero, hence, being left-invariant, it is a 
  left-Haar measure; second, we have that
  for each point $p\in N$
     there are 
     $c, r_0>0$ so that 
\begin{equation}\label{Ahlfors_IN_N.0}
c^{-1} r^{Q}\leq \mu_N \big( B_N(p,r)\big) \leq c  r^{Q},  \qquad \forall r\in (0,r_0).
\end{equation}

 \begin{proof}
 We plan to use Theorem~\ref{thm2}. Let $X=G$, $Y=G/H$, and $\pi:X\to Y$ the projection. We identify $G/H$ with $N$ and $\pi$ with $\pi_N.$   We shall show that  the   $Q$-Hausdorff    measure $\mu:=\mu_N$ on $N$ is such that  
 for every $r_0>0$ and every $x,x' \in G$ with $\pi (x)=\pi(x')$  there is $C>0$ such that 
   \begin{equation}\label{Ahlfors27ott.112finale}
\mu (\pi (B(x,r))) \leq C \mu (\pi (B(x',r))), \qquad \forall r\in (0,r_0).
\end{equation}

 
 Fix $r_0>0$ and $x,x' \in G$ such that $\pi (x)=\pi (x'),$ i.e., there is $n\in N$ and $h,h' \in H$ such that $x=nh, x'=nh'.$ Hence  \begin{equation*}
\begin{aligned}
\pi (B(x',r)) &  = \pi (\{ nh'g \,:\, g\in G, \, d(1,g ) \leq r \})\\
&  = \{ nC_{h'} (\pi(g)) \,:\,  g\in G, \, d(1,g ) \leq r \}\\
& = L_{n}C_{h'} (\pi(B(1,r))).
\end{aligned}
\end{equation*}
Moreover, using a similar argument, it easy to see that $\pi(B(1,r))= C_{h^{-1}} L_{n^{-1}} (\pi(B(x,r)))$ and, consequently,
 \begin{equation}\label{equation9.03}
\begin{aligned}
\pi (B(x',r)) & = L_{n}C_{h'}C_{h^{-1}} L_{n^{-1}} (\pi(B(x,r))), \qquad \forall r\in (0,r_0).
\end{aligned}
\end{equation} 

Since $\pi (\bar B(x,r_0))$ is contained in a compact set $K \subset N,$ we have that on the set $K,$ the map $N \ni m \mapsto  L_{n}C_{h'}C_{h^{-1}} L_{n^{-1}} (m)$ is smooth  and hence has bounded Jacobian with respect to  the (smooth) measure $\mu$, say by $C>0.$
 Hence, \eqref{Ahlfors27ott.112finale} holds and we apply Theorem~\ref{thm2}  in order to obtain the thesis.
%
%
%
\end{proof}
  
   As a consequence, as done by Franchi and Serapioni \cite{FS16}, one could see this result in the context of Carnot groups:
 \begin{coroll}[FSSC]
  Let $G=N  \rtimes H$ be a Carnot group that is the semidirect product of two   homogeneous subgroups, with $N$ normal.
 For every $\phi:N\to H$  intrinsically Lipschitz map in the FSSC sense, the set $\Gamma _ {\psi}$ is locally Ahlfors regular. 

 \end{coroll}

 \begin{proof} 
	In order to justify the application of Theorem \ref{theoremAhlfors}, we stress that  the distance $d$ on each Carnot group is boundedly compact and, since  $N$ is homogeneous, the distance $d$ restricted on $N$ is homogenous and hence $N$ is $Q$-Ahlfors regular.
	Because  on Carnot groups intrinsically Lipschitz maps in the FSSC sense are in correspondence (with same graphs) to intrinsically $L$-Lipschitz section (see Proposition~\ref{ideaCarnot}), Theorem \ref{theoremAhlfors} gives the corollary.
\end{proof}

\subsection{Level sets and extensions in groups}  In this section we present Theorem~\ref{thm3}  in Carnot groups which is already proved in \cite[Theorem 1.4]{Vittone20}.  We underline that  Vittone shows the result in $\R^s$ and not only in $\R$ and he uses the coercivity condition, which corresponds  to asking a  biLipschitz property of $f$ on the fibers.  However,  it is possible to obtain the  following result:
 \begin{theorem}[Vittone]\label{Level sets and extensions in groups} 
Let  $G=N \rtimes H$ be a Carnot  group that is the semidirect product of  a normal subgroup $N $ and a one-dimensional horizontal subgroup $H$, i.e.,  $H=\{ \exp (tX) \,: \, t\in \R \}$ for some $X$ in the first layer of $G$.  If $S \subset G$ is not empty, then the following statements are equivalent:
\begin{enumerate}
\item[(\ref{Level sets and extensions in groups}.1)] there exists a map $\psi : U\subseteq N \to H $ that is  intrinsically Lipschitz in the FSSC sense, here $U $ is a subset of $N$, with $S =\Gamma _\psi;$
\item[(\ref{Level sets and extensions in groups}.2)] there exists a Lipschitz map  $f:G \to \R$ that  is biLipschitz on fibers such that 
  $$ S \subset f^{-1} (0). $$ 
\end{enumerate}
\end{theorem}
 
  \begin{proof} 
  Recall from Theorem \ref{ideaCarnot} that there is a dual viewpoint between 
maps  $\psi : U\subseteq N \to H $ that are  intrinsically Lipschitz in the FSSC sense
  and maps
  $\phi :U\subseteq G/H \to G$ that are   intrinsically Lipschitz sections of the projection $\pi :G \to G/H.$
  We shall use this identification.
  
    The proof of the theorem will be just an application of our Theorem~\ref{thm3}. 
  We apply the theorem with the following notation: $X=G, Y= G/H\simeq N  $, $\pi :G \to G/H,$ $Z=\R$.
  
  [(\ref{Level sets and extensions in groups}.2)
  $\Rightarrow$
  (\ref{Level sets and extensions in groups}.1)] From Theorem~\ref{thm3}.i there is 
  an intrinsically  Lipschitz section  $\phi : G/H\to G $ (and equivalently a map $\psi:  N \to H $ that is  intrinsically Lipschitz in the FSSC sense)  such that
$
\Gamma _\psi  =\phi (G/H)=f^{-1} (0)$. Then, it is enough to take $U:=\{n\in N : n\psi(n) \in S\}$ and restrict the $\psi$ to $U$.

  [(\ref{Level sets and extensions in groups}.1)
  $\Rightarrow$
  (\ref{Level sets and extensions in groups}.2)]  
  Next we use Theorem~\ref{thm3}.ii. We have that $X=G$ is geodesic and that admits equivalent homogeneous distances $\rho$ with the property that the distance from the origin $1_G$ is smooth away from $\exp(V_1)$.  
  We also take  $\tau := \pi_H:G\to \R $, where we   identify $\R$ with $H$ via the map $t \mapsto \exp (tX).$ Since $ \tau$ can be seen as the projection modulo the normal subgroup $N$, then it is Lipschitz. 
  Moreover, the level sets $\tau ^{-1}(\tau_0)$ are left-translations of $N$, which are intrinsically $k$-Lipschitz graph of sections $\phi_{g_0}(gH):= g_0\pi_N(g)$, see Proposition \ref{prop7nov1021} together with \ref{ideaCarnot}.i.b.
  Next, we check the   assumption (2) of \ref{thm3}.ii. Because of left invariance, we can just consider the function
  $x \mapsto \delta _{0} (x) := \rho (1_G,  \pi_N (x))$ 
  on the set $\{|\pi_H (x)| \leq \delta _{0} (x)\}$.
  Notice that, denoting by $M_\eps$ the intrinsic multiplication in the Carnot group, we have 
  $ \pi_H (M_\eps(x)) = \eps \pi_H (x) $ and $
     \delta _{0} (M_\eps(x))= \eps \delta _{0} (x)$. Hence  the set $\{|\pi_H (x)| < \delta _{0} (x)\}$ is dilation invariant and its intersection avoids $\exp(V_1)$. Consequently, on it the function  $\delta _{0}$ is the composition of smooth functions, which are therefore Lipschitz on compact sets. Again, by homogeneity, the function is Lipschitz. (This last part of the argument is not very different from Vittone's original proof.)
     Applying Theorem~\ref{thm3}.ii concludes the existence of the requested function $f:G\to \R$.
	\end{proof}

 \bibliographystyle{alpha}
\bibliography{DDLD}

\end{document}